\documentclass[11pt,reqno]{amsart}

\usepackage{epsfig}
\usepackage{subfigure}
\usepackage{verbatim}
\usepackage[dvipsnames,usenames]{color} %\color{Red}
\usepackage{amssymb}
\usepackage{amsfonts}
\usepackage{amsbsy}
\usepackage{cite}
\usepackage[cmyk]{xcolor}
\usepackage{setspace}

\newcommand{\R}{{\mathbb R}}

\newcommand{\C}{{\mathbb C}}

\renewcommand{\eqref}[1]{(\ref{#1})}

\newtheorem{thm}{Theorem}[section]
\newtheorem{theo}[thm]{Theorem}

\newtheorem{coro}[thm]{Corollary}

\newtheorem{lem}[thm]{Lemma}

\newtheorem{prop}[thm]{Proposition}
\theoremstyle{definition}
\newtheorem{exam}{Example}[section]
\newtheorem{defi}{Definition}[section]
\newtheorem{rmk}[thm]{Remark}
\numberwithin{equation}{section}

\def\vint{\mathop{\mathchoice%
         {\setbox0\hbox{$\displaystyle\intop$}\kern 0.22\wd0%
          \vcenter{\hrule width 0.6\wd0}\kern -0.82\wd0}%
         {\setbox0\hbox{$\textstyle\intop$}\kern 0.2\wd0%
          \vcenter{\hrule width 0.6\wd0}\kern -0.8\wd0}%
         {\setbox0\hbox{$\scriptstyle\intop$}\kern 0.2\wd0%
          \vcenter{\hrule width 0.6\wd0}\kern -0.8\wd0}%
         {\setbox0\hbox{$\scriptscriptstyle\intop$}\kern 0.2\wd0%
          \vcenter{\hrule width 0.6\wd0}\kern -0.8\wd0}}%
         \mathopen{}\int}
\begin{document}

\title[Nodal sets and continuity of eigenfunctions]
{Nodal sets and continuity of eigenfunctions of Kre$\mathbf{\breve{{\i}}}$n-Feller operators}

\author[S.-M. Ngai]{Sze-Man Ngai}
\address{Key Laboratory of High Performance Computing and Stochastic Information
Processing (HPCSIP) (Ministry of Education of China), College of
Mathematics and Statistics, Hunan Normal University, Changsha, Hunan
410081, China, and Department of Mathematical Sciences,  Georgia Southern
University, Statesboro, GA 30460-8093, USA.}
\email{smngai@georgiasouthern.edu}

\author[M.-K. Zhang]{Meng-Ke Zhang}
\address{Key Laboratory of High Performance Computing and Stochastic Information
Processing (HPCSIP) (Ministry of Education of China), College of
Mathematics and Statistics, Hunan Normal University, Changsha, Hunan
410081, China.}
\email{2750600901@qq.com}

\author[W.-Q. Zhao]{Wen-Quan Zhao}
\address{Key Laboratory of High Performance Computing and Stochastic Information
Processing (HPCSIP) (Ministry of Education of China), College of
Mathematics and Statistics, Hunan Normal University, Changsha, Hunan
410081, China.}
\email{zhaowq1008@hunnu.edu.cn}
%\thanks{}

\thanks{The authors are supported in part by the National Natural Science Foundation of China, grants 12271156 and 11771136, and Construct Program of Key Discipline in Hunan Province. The first author is also supported in part by a Faculty Research Scholarly Pursuit Funding from Georgia Southern University.}

\begin{abstract} Let $\mu$ be a compactly supported positive finite Borel measure on $\R^{d}$.  Let $0<\lambda_{1}\leq\lambda_{2}\leq\ldots$ be eigenvalues of the Kre$\breve{{\i}}$n-Feller operator $\Delta_{\mu}$. We prove that, on a bounded domain, the nodal set of a continuous $\lambda_{n}$-eigenfunction of a Kre$\breve{{\i}}$n-Feller operator divides the domain into at least 2 and at most $n+r-1$ subdomains, where $r$ is the multiplicity of $\lambda_{n}$. This work generalizes the nodal set theorem of the classical Laplace operator to Kre$\breve{{\i}}$n-Feller operators on bounded domains. We also prove that on bounded domains on which the classical Green function exists, the eigenfunctions of a Kre$\breve{{\i}}$n-Feller operator are continuous.
\end{abstract}

\date{\today}
\subjclass[2010]{Primary: 35J05, 35B05, 34L10; Secondary: 28A80, 35J08. }
\keywords{Kre$\breve{{\i}}$n-Feller operators; nodal set; continuous eigenfunctions}
\maketitle

\section{Introduction}\label{sec1}
For a bounded domain (i.e., an open and connected set) $\Omega\subset\R^{d}$, consider the following Dirichlet problem
\begin{equation}\label{1}
\begin{cases}~-\Delta u(\boldsymbol{x})=\lambda u(\boldsymbol{x}),&\boldsymbol{x}\in\Omega,\\~ u(\boldsymbol{x})=0,&\boldsymbol{x}\in\partial\Omega,\end{cases}
\end{equation}
with eigenvalue $\lambda$ and eigenfunction $u$. The {\em nodal set} of $u(\boldsymbol{x})$ is defined as
$$\mathcal{Z}(u):=\{\boldsymbol{x}\in\Omega:u(\boldsymbol{x})=0\}.$$
It is known that  the eigenvalues can be ordered as
$$0<\lambda_{1}\leq\lambda_{2}\leq\cdots$$
with $\lim_{n\to\infty}\lambda_{n}=\infty$.  Properties of nodal set of the eigenfunctions have been studied extensively
(see \cite{Alessandrini_1994,Alessandrini_1998,Courant-Hilbert_1953,Gladwell-Zhu_2002,Lin_1987,Logunov_2018,Melas_1992,Payne_1967,Payne_1973,Pleijel_1956,Strauss_2008}
and references therein). Let $\Omega\subset\R^{d}$ be a bounded domain. By a {\em$\lambda$-eigenfunction} we mean an eigenfunction corresponding to the eigenvalue $\lambda$.
The Courant nodal domain theorem says that the nodal set of a $\lambda_{n}$-eigenfunction of (\ref{1}) divides $\Omega$ into
  at most $n$ subdomains (see e.g. \cite{Courant-Hilbert_1953}). Gladwell and Zhu \cite{Gladwell-Zhu_2002} studied the nodal sets of eigenfunctions of the Helmholtz equation
$$\Delta u(\boldsymbol{x})+\lambda\rho(\boldsymbol{x})u(\boldsymbol{x})=0, \qquad\boldsymbol{x}\in\Omega,$$
where $\Omega\subset\R^{d}$ is a domain with $d\geq1$ and $\rho(\boldsymbol{x})$ is positive and bounded. They proved that the nodal
set of a $\lambda_{n}$-eigenfunction divides $\Omega$ into at most $n+r-1$ subdomains, where $r$ is the multiplicity of $\lambda_{n}$. One goal of this paper is to generalize the above theorem to Laplace operators defined by measures (see definition in Section 2) on a domain $\Omega\subset\R^{d}$. Such operators are also called {\em Kre\u{\i}n-Feller operators} and are introduced in \cite{Feller_1957,Kac-Krein_1958,Krein_1952}. These operators are used to describe physical phenomena, such as wave propagation or heat conduction, in media with an inhomogeneous mass distribution modeled by a measure $\mu$, such as a fractal measure.

Kre{\rm$\breve{{\i}}$}n-Feller operators have been studied extensively. McKean and Ray \cite{Mckean-Ray_1962} studied spectral asymptotics of Kre{\rm$\breve{{\i}}$}n-Feller operator defined by the Cantor measure. Freiberg \cite{Freiberg_2003} studied analytic properties of the operators defined on the line. Hu {\em et al} \cite{Hu-Lau-Ngai_2006} studied spectral properties of Kre{\rm$\breve{{\i}}$}n-Feller operator defined on a domain of $\R^{d}$. Deng and Ngai \cite{Deng-Ngai_2015}, Pinasco and Scarola \cite{Pinasco-Scarola_2019} studied the eigenvalue estimates of such operators. Kesseb\"{o}hmer and Niemann \cite{Kessebohmer-Niemann_2022,Kessebohmer-Niemann_2022A} studied the relation between the $L^{q}$-spectrum and spectral dimension of such an operator. For additional work  associated with these operators, including eigenvalues and eigenfunctions, spectral asymptotics, spectral gaps, spectral dimensions, wave equation, heat equation and heat kernel estimates, the reader is referred to  \cite{Bird-Ngai-Teplyaev_2003,Chan-Ngai-Teplyaev_2015,Chen-Ngai_2010,Deng-Ngai_2015,Deng-Ngai_2021,Freiberg_2003,Freiberg_2011,Freiberg-Zahle_2002,Gu-Hu-Ngai_2020,Hu-Lau-Ngai_2006, Kessebohmer-Niemann_2022,Kessebohmer-Niemann_2022A,Naimark-Solomyak_1994,Naimark-Solomyak_1995,Ngai_2011,Ngai-Tang-Xie_2018,Ngai-Xie_2020,Ngai-Xie_2021,
Pinasco-Scarola_2019,Pinasco-Scarola_2021,Tang-Ngai_2022} and references therein.

We will summarize the definition of a Kre{\rm$\breve{{\i}}$}n-Feller operator in Section $2$. We denote by $\Delta_{\mu}$ the Kre{\rm$\breve{{\i}}$}n-Feller operator defined by a measure $\mu$ (definition see Section 2). In this article, we consider eigenvalues and eigenfunctions associated with the following Dirichlet problem:
\begin{equation}\label{eq(2.2)}
\begin{cases}~-\Delta_{\mu}u(\boldsymbol{x})=\lambda u(\boldsymbol{x}),&\boldsymbol{x}\in\Omega,\\~ u(\boldsymbol{x})=0,&\boldsymbol{x}\in\partial\Omega.\end{cases}
\end{equation}
Let $\underline{\dim}_{\infty}(\mu)$ be defined as in (\ref{dim}). It is shown in \cite[Theorem 1.2]{Hu-Lau-Ngai_2006} that, under the assumption $\underline{\dim}_{\infty}(\mu)>d-2$,
there exists an orthonormal basis $\{\phi_{n}\}_{n=1}^{\infty}$ of $L^{2}(\Omega,\mu)$  consisting of (Dirichlet) eigenfunctions of $\Delta_{\mu}$.
The eigenvalues $\{\lambda_{n}\}_{n=1}^{\infty}$ satisfy $0<\lambda_{1}\leq\lambda_{2}\leq\cdots$ and $\lim_{n\to\infty}\lambda_{n}=\infty$. We let
\begin{equation}\label{defiZ}
\mathcal{Z}_{\mu}(u):=\{\boldsymbol{x}\in\Omega: u(\boldsymbol{x})=0\}
\end{equation}
be the nodal set of an eigenfunction $u$ of $\Delta_{\mu}$. Under the assumption of the continuity of eigenfunctions, we have the following theorem.

\begin{theo}\label{T1}
Let $\Omega\subset\mathbb{R}^{d}\,(d\geq1)$ be a bounded domain and $\mu$ be a positive finite Borel measure on $\mathbb{R}^{d}$
with ${\rm supp}(\mu)\subset \overline{\Omega}~and~\mu(\Omega)>0$. Assume $\underline{\dim}_{\infty}(\mu)>d-2$.
Let the eigenvalues $\{\lambda_{n}\}_{n=1}^{\infty}$ of (\ref{eq(2.2)}) be arranged in an increasing order and
let $u_{n}(\boldsymbol{x})$ be a $\lambda_{n}$-eigenfunction. Suppose $u_{n}\in C(\overline{\Omega})$. Then
\begin{enumerate}
\item[(a)]~$u_{1}(\boldsymbol{x})$ is nonzero in $\Omega$.

\item[(b)]~For $n\geq 2$, if $\lambda_{n}$ has multiplicity $r\geq1$, then $\mathcal{Z}_{\mu}(u_{n})$ divides $\Omega$ into at least $2$ and at most $n+r-1$ subdomains.
\end{enumerate}
\end{theo}
In order to prove this result, we need the maximum principle of continuous $\mu$-subharmonic functions (Definition \ref{Def3.2}) which we will prove in Section 3.

Note that the definition of {\color{blue}a} nodal set $\mathcal{Z}_{\mu}(u)$ in (\ref{defiZ}) makes sense only if $u$ is defined everywhere and not just almost everywhere. As the domain of $\Delta_{\mu}$ consists of Sobolev functions, we need to study the continuity of the eigenfunctions of $\Delta_{\mu}$. It is known that for the classical Laplacian, the eigenfunctions, in fact the functions in ${\rm{Dom}}(\Delta)$, are all continuous and differentiable. If the measure $\mu$ is absolutely continuous with respect to Lebesgue measure, then the eigenfunctions of $\Delta_{\mu}$ are continuous (see e.g. \cite{Evans_2010}). On $\R$, the eigenfunctions of $\Delta_{\mu}$ are continuous, since Sobolev functions are continuous. To the best of the authors' knowledge, the continuity of eigenfunctions in higher dimensions is unknown in general. Motivated by the requirement of the continuity of eigenfunctions in the definition of $\mathcal{Z}_{\mu}(u)$ and in Theorem \ref{T1}, we prove the following main theorem. We refer to the definition of Lipschitz boundary in Definition \ref{defibdary}.

\begin{theo}\label{thm2}
Let $\Omega$ be a bounded domain in ${\R}^d$ on which the classical Green function $G(x,y)$ exists and let $\mu$ be a finite positive Borel measure with {\rm supp}$(\mu)\subseteq\overline{\Omega}$
and $\mu(\Omega)>0$. Assume $\underline{\dim}_{\infty}(\mu)>d-2$. Then the eigenfunctions of $\Delta_{\mu}$ are continuous on $\Omega$. Moreover, if additionally, $\Omega$ has Lipschitz boundary, then the eigenfunctions are continuous on $\overline{\Omega}$.
\end{theo}

To prove this theorem, we apply the inverse operator of $-\Delta_{\mu}$, called {\color{blue}the} Green operator (see Section 5), which is defined by the classical Green function. By expressing the eigenfunctions of $\Delta_{\mu}$ in terms of Green operator and using the continuity of {\color{blue}the} Green function, we prove Theorem \ref{thm2}; details are given in Section 5. This theorem shows that the Dirichlet problem (\ref{eq(2.2)}) has continuous solutions on a bounded domain $\Omega\subset\R^{d}$ on which the Green function exists.

This paper is organized as follows. In Section 2, we summarize the definition of $\Delta_{\mu}$ in \cite{Hu-Lau-Ngai_2006}. In Section 3, we study the properties of $\Delta_{\mu}$ and prove the maximum principle of continuous $\mu$-subharmonic functions. Section 4 and 5 are devoted to the proofs of Theorems \ref{T1} and  \ref{thm2}, respectively. In Section 6, we construct examples of continuous eigenfunctions corresponding singular measures in $\R^{2}$.
\maketitle

\section{ Preliminaries }\label{sec2}
In this section, we will summarize the definition of Kre{\rm$\breve{{\i}}$}n-Feller operators; details can be found in \cite{Hu-Lau-Ngai_2006,Deng-Ngai_2021}.
Let $\Omega\subset\R^{d}$ be a bounded domain and $\mu$ be a finite positive Borel measure with supp$(\mu)\subseteq\overline{\Omega}$,
where $\overline{\Omega}$ is the closure of $\Omega$. Let $\partial\Omega:=\overline{\Omega}\backslash\Omega$ be the boundary of $\Omega$.
Let $d\boldsymbol{x}$ be the Lebesgue measure on $\R^{d}$, and let $H^{1}(\Omega)$ be the Sobolev space equipped with the inner product
$$\langle u, v\rangle_{H^{1}(\Omega)}:=\int_{\Omega}uv\,d\boldsymbol{x}+\int_{\Omega}\nabla u\cdot\nabla v\,d\boldsymbol{x}.$$
Let $H_{0}^{1}(\Omega)$ be the completion of $C_{c}^{\infty}(\Omega)$ under the above inner product, where $C_{c}^{\infty}(\Omega)$ is the space of all smooth functions with compact support in $\Omega$. Let $L^{2}(\Omega,~\mu)$ be the space of all square integrable functions with respect to $\mu$. The norm in $L^{2}(\Omega,~\mu)$ is defined by
$$\|u\|_{L^{2}(\Omega,~\mu)}:=\left(\int_{\Omega}|u|^{2}\,d\mu\right)^{1/2}.$$
Throughout this paper, we write $L^{2}(\Omega):=L^{2}(\Omega, d\boldsymbol{x})$. The {\em lower $L^{\infty}$-dimension} of $\mu$ is defined as
\begin{equation}\label{dim}
\underline{\rm{dim}}_{\infty}(\mu):=\liminf_{\delta\to0^{+}}\frac{\ln(\sup_{\boldsymbol{x}}\mu(B_{\delta}(\boldsymbol{x})))}{\ln{\delta}},
\end{equation}
where $B_{\delta}(\boldsymbol{x})$ is the ball with center $\boldsymbol{x}$ and radius $\delta$ and the supremum is taken over all $\boldsymbol{x}\in\rm{supp}(\mu)$ (see \cite{Strichartz_1993} for details). In order to define Kre{\rm$\breve{{\i}}$}n-Feller operators, we need the following assumption, which is called the Poincar\'e inequality associated with measures (MPI): there exists a constant $C>0$ such that
$$\int_{\Omega}|u|^{2}\,d\mu\leq C\int_{\Omega}|\nabla u|^{2}\,d\boldsymbol{x},\qquad \text{for~all~$u\in C_{c}^{\infty}(\Omega)$}.$$
We know that if $\underline{\rm{dim}}_{\infty}(\mu)>d-2$, then $\mu$ satisfies $($MPI$)$ (see \cite[Theorem 1.1]{Hu-Lau-Ngai_2006}). $($MPI$)$ implies that each equivalence class $u\in H_{0}^{1}(\Omega)$ contains a unique (in $L^{2}(\Omega,\mu)$ sense) member $\hat{u}\in L^{2}(\Omega,~\mu)$ that satisfies the following two conditions:

\noindent$(a)$ there exists a sequence $\{u_{n}\}$ in $C_{c}^{\infty}(\Omega)$ such that $u_{n}\to\hat{u}$ in $H_{0}^{1}(\Omega)$ and $u_{n}\to\hat{u}$ in $L^{2}(\Omega,\mu)$;

\noindent$(b)$ $\hat{u}$ satisfies the $($MPI$)$.

Such $\hat{u}$ is called the {\em $L^{2}(\Omega,\mu)$-representative} of $u$. Under the assumption $($MPI$)$, we define a map $\mathcal{I}$: $H_{0}^{1}(\Omega)\to L^{2}(\Omega,\mu)$ by
\begin{equation}\label{eq(2.1.1)}
\mathcal{I}(u)=\hat{u}.
\end{equation}
The mapping $\mathcal{I}$ is in general not injective. We define the following closed subset of $H_{0}^{1}(\Omega)$:
$$\mathcal{N}:=\big\{u\in H_{0}^{1}(\Omega):~\|\mathcal{I}(u)\|_{L^{2}(\Omega,\mu)}=0\big\}.$$
Let $\mathcal{N}^{\bot}$ be the orthogonal complement of $\mathcal{N}$ in $H_{0}^{1}(\Omega)$. Then $\mathcal{I}:\mathcal{N}^{\bot}\to L^{2}(\Omega,~\mu)$ is injective. We will denote $\hat{u}$ simply by $u$ if there is no confusion possible.

Consider a nonnegative bilinear form $\mathcal{E}(\cdot,\cdot)$ in $L^{2}(\Omega,\mu)$ defined as
\begin{equation}\label{(2.1)}
\mathcal{E}(u,v):=\int_{\Omega}\nabla u\cdot\nabla v\,d\boldsymbol{x},
\end{equation}
with Dom$(\mathcal{E})=\mathcal{N}^{\bot}$. (MPI) implies that $(\mathcal{E},\rm{Dom}(\mathcal{E}))$ is a closed quadratic form on $L^{2}(\Omega,\mu)$ (see \cite[Proposition 2.1]{Hu-Lau-Ngai_2006}).
Hence, there exists a nonnegative self-adjoint operator $-\Delta_{\mu}$ such that
$${\rm Dom}(\mathcal{E})={\rm Dom}((-\Delta_{\mu})^{1/2})$$
and
$$\mathcal{E}(u,v)=\big\langle(-\Delta_{\mu})^{1/2}u ,~(-\Delta_{\mu})^{1/2}v\big\rangle_{L^{2}(\Omega,\mu)},\qquad \text{for~all~$u,v\in\rm{Dom}(\mathcal{E})$}$$
(see \cite{Davies_1995}), where the $\langle\cdot,\cdot\rangle_{L^{2}(\Omega,\mu)}$ is the inner product in $L^{2}(\Omega,\mu)$.
 We call the above $\Delta_{\mu}$ the {\em(Dirichlet) Laplacian} with respect to $\mu$ or the {\em Kre\u{\i}n-Feller operator} defined by $\mu$.
It follows from \cite[Proposition 2.2]{Hu-Lau-Ngai_2006} that
$u\in\rm{Dom}(\Delta_{\mu})$ and $-\Delta_{\mu}u=f$ if and only if $-\Delta u=f\,d\mu$ in the sense of distribution, i.e.,
\begin{equation}
\int_{\Omega}\nabla u\cdot\nabla \varphi\,d\boldsymbol{x}=\int_{\Omega}f\varphi\,d\mu, \qquad~\text{for all}~\varphi\in C_{c}^{\infty}(\Omega).
\end{equation}

\section{ maximum principle }\label{sec3}
In this section, we prove the maximum principle for a continuous $\mu$-subharmonic function. We first define $\mu$-subharmonic functions.

\begin{defi}\label{Def3.2}
We call $u\in {\rm Dom}(\Delta_{\mu})$ a {\em$\mu$-subharmonic function} if $\Delta_{\mu}u\geq0$ ($\mu$-a.e.). Call $u\in {\rm Dom}(\Delta_{\mu})$ a {\em$\mu$-superharmonic function} if $\Delta_{\mu}u\leq0$ ($\mu$-a.e.).
\end{defi}
Note that the class of $\mu$-subharmonic (resp. $\mu$-superharmonic) functions and the class of classical subharmonic (resp. superharmonic) functions are in general not equal. In fact, the function $u$ in
Example \ref{EX1} is $\mu$-superharmonic but not superharmonic. Nevertheless, we will prove that the maximum principle still holds for continuous $\mu$-subharmonic (resp. $\mu$-superharmonic) functions. In order to prove this, we use mollifiers.

Let $\Omega\subset\mathbb{R}^{d}$ be a bounded domain and let $u\in H_{0}^{1}(\Omega)$. Let $\widetilde{u}$ be the zero-extension of $u$, i.e.,
$$\widetilde{u}(\boldsymbol{x})=\begin{cases} u(\boldsymbol{x}) &\text{if}~\boldsymbol{x}\in\Omega, \\
 0 &\text{if}~\boldsymbol{x}\in\R^{n}\setminus\Omega.\end{cases}$$
It is known that $\widetilde{u}\in H^{1}(\R^{d})$ (see, e.g.,\cite[Lemma 3.27]{Adams-Fournier_2003}). Let $\epsilon>0$. Define
\begin{equation}\label{eq(3.8)}
\widetilde{u}^{\epsilon}:=\eta_{\epsilon}\ast\widetilde{u},
\end{equation}
where $\eta_{\epsilon}\geq0$ are mollifiers. It is known that for each $\epsilon>0$, $\eta_{\epsilon}$ is smooth and satisfies $\int_{\R^{d}}\eta_{\epsilon}(\boldsymbol{x})\,d\boldsymbol{x}=1$. Moreover,  ${\rm\text{supp}}(\eta_{\epsilon}(\boldsymbol{x}))\subset B_{\epsilon}(\boldsymbol{0})$ and
\begin{equation}\label{etasupp}
{\rm supp}(\eta_{\epsilon}(\boldsymbol{x}-\cdot))\subset B_{\epsilon}(\boldsymbol{x})
\end{equation}
(see e.g.,\cite{Evans_2010,Adams-Fournier_2003} for details).
The following proposition follows from \cite[Theorem 2.29]{Adams-Fournier_2003}.

\begin{prop}\label{pro2.1}
Let $\widetilde{u}^{\epsilon}$ be defined in (\ref{eq(3.8)}). Then
\begin{enumerate}
\item[(a)] $\widetilde{u}^{\epsilon}\in C^{\infty}(\R^{d})$.
\item[(b)] If $u\in C(\overline{\Omega})$, then $\widetilde{u}^{\epsilon}\to u$ uniformly on $\Omega$.
\item[(c)] $\widetilde{u}^{\epsilon}\to u$ in $H^{1}(\Omega)$ as $\epsilon\to0$.
\end{enumerate}
\end{prop}

\begin{prop}\label{pro2.2b}
Let $\Omega\subset\R^{d}$ $(d\geq1)$ be a bounded domain and $\mu$ be a positive finite Borel measure on $\mathbb{R}^{d}$
with ${\rm supp}(\mu)\subset \overline{\Omega}~and~\mu(\Omega)>0$. Assume that (MPI) holds. Let $u\in {\rm Dom}(\Delta_{\mu})$ be a $\mu$-subharmonic function and $\widetilde{u}^{\epsilon}$ be defined as in (\ref{eq(3.8)}). Let $\boldsymbol{z}\in\Omega$. Then for any $0<r<{\rm dist}(\boldsymbol{z},\partial\Omega)$,
$$\lim_{\epsilon\to0}\int_{B_{r}(\boldsymbol{z})}\Delta(\widetilde{u}^{\epsilon}|_{\Omega})(\boldsymbol{x})\,d\boldsymbol{x}=\int_{B_{r}(\boldsymbol{z})}\Delta_{\mu}u(\boldsymbol{x})\,d\mu.$$
\end{prop}
\begin{proof}
Let $\boldsymbol{z}\in\Omega$ and arbitrary $\epsilon>0$ sufficiently small such that $\epsilon<{\rm\text{dist}}(\boldsymbol{z},\partial\Omega)/4$. Let $r\in(\epsilon,{\rm\text{dist}}(\boldsymbol{z},\partial\Omega)-3\epsilon)$. Then for $x\in B_{r}(\boldsymbol{z})$, by Proposition \ref{pro2.1}(a), we have
\begin{align*}
\int_{B_{r}(\boldsymbol{z})}\Delta(\widetilde{u}^{\epsilon}|_{\Omega})(\boldsymbol{x})\,d\boldsymbol{x}
&=\int_{B_{r}(\boldsymbol{z})}\Delta\big((\eta_{\epsilon}\ast\widetilde{u})|_{\Omega}\big)(\boldsymbol{x})\,d\boldsymbol{x}\\
&=\int_{B_{r}(\boldsymbol{z})}(\Delta\eta_{\epsilon}\ast\widetilde{u})|_{\Omega}(\boldsymbol{x})\,d\boldsymbol{x}\\
&=\int_{B_{r}(\boldsymbol{z})}\int_{B_{\epsilon}(\boldsymbol{x})}\Delta\eta_{\epsilon}(\boldsymbol{x}-\boldsymbol{y})\widetilde{u}(\boldsymbol{y})\,d\boldsymbol{y}d\boldsymbol{x}.\qquad\qquad(\text{by}~(\ref{etasupp}))
\end{align*}
Since (\ref{etasupp}) and $\eta_{\epsilon}(\boldsymbol{x}-\boldsymbol{y})=0$ on $\partial B_{\epsilon}(\boldsymbol{x})$ (see, e.g.,\cite{Adams-Fournier_2003,Evans_2010}), we have
\begin{align*}
\int_{B_{r}(\boldsymbol{z})}\Delta\widetilde{u}^{\epsilon}|_{\Omega}(\boldsymbol{x})\,d\boldsymbol{x}
&=-\int_{B_{r}(\boldsymbol{z})}\int_{B_{\epsilon}(\boldsymbol{x})}\nabla\eta_{\epsilon}(\boldsymbol{x}-\boldsymbol{y})\cdot\nabla\widetilde{u}(\boldsymbol{y})\,d\boldsymbol{y}d\boldsymbol{x}\\
&=-\int_{B_{r}(\boldsymbol{z})}\int_{B_{\epsilon}(\boldsymbol{x})}\nabla\eta_{\epsilon}(\boldsymbol{x}-\boldsymbol{y})\cdot\nabla u(\boldsymbol{y})\,d\boldsymbol{y}d\boldsymbol{x}\qquad\qquad(\widetilde{u}=u~\text{in}~B_{\epsilon}(\boldsymbol{x}))\\
&=\int_{B_{r}(\boldsymbol{z})}\int_{B_{\epsilon}(\boldsymbol{x})}\eta_{\epsilon}(\boldsymbol{x}-\boldsymbol{y})\Delta_{\mu}u(\boldsymbol{y})\,d\mu(\boldsymbol{y})d\boldsymbol{x}.
\end{align*}
The last equality follows by \cite[Proposition 2.2]{Hu-Lau-Ngai_2006} and (\ref{etasupp}). Therefore, if we let $\chi_{B_{\epsilon}(\boldsymbol{x})}$ be the characteristic function on $B_{\epsilon}(\boldsymbol{x})$, then
\begin{align}\label{eqxy}
\int_{B_{r}(\boldsymbol{z})}\Delta\widetilde{u}^{\epsilon}|_{\Omega}(\boldsymbol{x})\,d\boldsymbol{x}
&=\int_{B_{r}(\boldsymbol{z})}\int_{B_{r+2\epsilon}(\boldsymbol{z})}\chi_{B_{\epsilon}(\boldsymbol{x})}(\boldsymbol{y})\eta_{\epsilon}(\boldsymbol{x}-\boldsymbol{y})\Delta_{\mu}u(\boldsymbol{y})\,d\mu(\boldsymbol{y})d\boldsymbol{x}\notag\\
&=\int_{B_{r+2\epsilon}(\boldsymbol{z})}\int_{B_{r}(\boldsymbol{z})}\chi_{B_{\epsilon}(\boldsymbol{x})}(\boldsymbol{y})\eta_{\epsilon}(\boldsymbol{x}-\boldsymbol{y})\Delta_{\mu}u(\boldsymbol{y})\,d\boldsymbol{x}d\mu(\boldsymbol{y})\qquad\qquad({\rm Fubini})\notag\\
&=\int_{B_{r+2\epsilon}(\boldsymbol{z})}\int_{B_{r}(\boldsymbol{z})\cap B_{\epsilon}(\boldsymbol{y})}\chi_{B_{\epsilon}(\boldsymbol{y})}(\boldsymbol{x})\eta_{\epsilon}(\boldsymbol{x}-\boldsymbol{y})\Delta_{\mu}u(\boldsymbol{y})\,d\boldsymbol{x}d\mu(\boldsymbol{y}) \notag\\
&\leq\int_{B_{r+2\epsilon}(\boldsymbol{z})}\Big(\int_{B_{\epsilon}(\boldsymbol{y})}\eta_{\epsilon}(\boldsymbol{x}-\boldsymbol{y})\,d\boldsymbol{x}\Big)\Delta_{\mu}u(\boldsymbol{y})\,d\mu(\boldsymbol{y})\notag\\
&=\int_{B_{r+2\epsilon}(\boldsymbol{z})}\Delta_{\mu}u(\boldsymbol{y})\,d\mu(\boldsymbol{y}).
\end{align}
On the other hand, since $\eta_{\epsilon}\geq0$ and $\Delta_{\mu}u\geq0$ $\mu$-a.e., we have
\begin{align}\label{eqdy}
\int_{B_{r}(\boldsymbol{z})}\Delta\widetilde{u}^{\epsilon}|_{\Omega}(\boldsymbol{x})\,d\boldsymbol{x}&=\int_{B_{r+2\epsilon}(\boldsymbol{z})}\int_{B_{r}(\boldsymbol{z})}\chi_{B_{\epsilon}(\boldsymbol{x})}(\boldsymbol{y})\eta_{\epsilon}(\boldsymbol{x}-\boldsymbol{y})\Delta_{\mu}u(\boldsymbol{y})\,d\boldsymbol{x}d\mu(\boldsymbol{y})\notag\\
&\geq\int_{B_{r-\epsilon}(\boldsymbol{z})}\int_{B_{r}(\boldsymbol{z})}\chi_{B_{\epsilon}(\boldsymbol{x})}(\boldsymbol{y})\eta_{\epsilon}(\boldsymbol{x}-\boldsymbol{y})\Delta_{\mu}u(\boldsymbol{y})\,d\boldsymbol{x}d\mu(\boldsymbol{y})\notag\\
&=\int_{B_{r-\epsilon}(\boldsymbol{z})}\Big(\int_{B_{\epsilon}(\boldsymbol{y})}\eta_{\epsilon}(\boldsymbol{x}-\boldsymbol{y})\,d\boldsymbol{x}\Big)\Delta_{\mu}u(\boldsymbol{y})\,d\mu(\boldsymbol{y})\notag\\
&=\int_{B_{r-\epsilon}(\boldsymbol{z})}\Delta_{\mu}u(\boldsymbol{y})\,d\mu(\boldsymbol{y}).
\end{align}
Combining (\ref{eqxy}) and (\ref{eqdy}), we have
$$\int_{B_{r-\epsilon}(\boldsymbol{z})}\Delta_{\mu}u(\boldsymbol{y})\,d\mu(\boldsymbol{y})\leq\int_{B_{r}(\boldsymbol{z})}\Delta\widetilde{u}^{\epsilon}|_{\Omega}(\boldsymbol{x})\,d\boldsymbol{x}\leq\int_{B_{r+2\epsilon}(\boldsymbol{z})}\Delta_{\mu}u(\boldsymbol{y})\,d\mu(\boldsymbol{y})$$
Letting $\epsilon\to0$ completes the proof.
\end{proof}

\begin{rmk}\label{rmk3.3}
We know that $\Delta_{\mu}u\in L^{2}(\Omega,\mu)\subset L^{1}(\Omega,\mu)$ for $u\in{\rm Dom}(\Delta_{\mu})$ (see \cite[Proposition 2.2]{Hu-Lau-Ngai_2006}). Therefore, if $u$ is a $\mu$-subharmonic function, then for any $B_{r}(\boldsymbol{z})\subset\Omega$, the limit
$$\lim_{\epsilon\to0}\int_{B_{r}(\boldsymbol{z})}\Delta\widetilde{u}^{\epsilon}|_{\Omega}(\boldsymbol{x})\,d\boldsymbol{x}=\int_{B_{r}(\boldsymbol{z})}\Delta_{\mu}u(\boldsymbol{x})\,d\mu$$
is nonnegative and finite.
\end{rmk}

We mention that in \cite[Theorem 3.3]{Tang-Ngai_2022}, it is shown that if $u\in C^{2}(\Omega)$ is a $\mu$-subharmonic function, then the maximum principle holds. The following theorem generalizes this theorem to $u\in C(\Omega)$.

\begin{theo}\label{thm1}
Let $\Omega\subset\mathbb{R}^{d}\,(d\geq1)$ be a bounded domain and $\mu$ be a positive finite Borel measure on $\mathbb{R}^{d}$
with ${\rm supp}(\mu)\subset \overline{\Omega}~and~\mu(\Omega)>0$. Assume that (MPI) holds. If $u\in C(\Omega)$ is a nonconstant $\mu$-subharmonic function, then $u$ cannot attain its maximum value in $\Omega$.
\end{theo}\begin{proof}
Some basic derivations are similar to the proof of the mean-value formula (see e.g., \cite[\S2.2.2 Theorem 2]{Evans_2010}); we will omit some details. Let $u$ be a nonconstant continuous $\mu$-subharmonic function. For any fixed $\boldsymbol{x}\in\Omega$, let $\widetilde{u}^{\epsilon}$ be defined as in (\ref{eq(3.8)}) and let $r>0$ be sufficiently small such that $B_{r}(\boldsymbol{x})\subset\Omega$. Define
$$\varphi_{\epsilon}(r):=\vint_{\partial B_{r}(\boldsymbol{x})}
\widetilde{u}^{\epsilon}|_{\Omega}(\boldsymbol{y})\,dS(\boldsymbol{y})=\vint_{\partial B_{1}(\mathbf{0})}
\widetilde{u}^{\epsilon}|_{\Omega}(\boldsymbol{x}+r\boldsymbol{z})\,dS(\boldsymbol{z}),$$
where
\begin{equation}\label{eqpj1}
\vint_{\partial B_{r}(\boldsymbol{x})}f\,dS:=\frac{1}{n\alpha(n)r^{n-1}}\int_{\partial B_{r}(\boldsymbol{x})}f\,dS
\end{equation}
is the average of $f$ over the sphere $\partial B_{r}(\boldsymbol{x})$ and $\alpha(n)=\pi^{n/2}\big/\Gamma(n/2+1)$ is the volume of
the unit ball $B_{1}(\mathbf{0})$ in $\R^{n}$. Let
\begin{equation}\label{eqpj2}
\vint_{B_{r}(\boldsymbol{x})}f\,d\boldsymbol{y}:=\frac{1}{n\alpha(n)r^{n}}\int_{B_{r}(\boldsymbol{x})}f\,d\boldsymbol{y}
\end{equation}
be the average of $f$ over $B_{r}(\boldsymbol{x})$ (see \cite[Appendix A]{Evans_2010}). Then
\begin{align*}
\varphi_{\epsilon}'(r)&=\vint_{\partial B_{1}(\mathbf{0})}D\widetilde{u}^{\epsilon}|_{\Omega}(\boldsymbol{x}+r\boldsymbol{z})\cdot \boldsymbol{z}\,dS(\boldsymbol{z})\\
&=\frac{1}{n\alpha(n)r^{n-1}}\vint_{B_{r}(\boldsymbol{x})}\Delta(\widetilde{u}^{\epsilon}|_{\Omega})(\boldsymbol{y})\,d\boldsymbol{y}.
\end{align*}
By Remark \ref{rmk3.3}, $\lim_{\epsilon\to0}\varphi_{\epsilon}'(r)\geq0$. It follows that for each $t>0$, there exists $\delta_{t}>0$ such that for all $\epsilon\in(0,\delta_{t})$, $\varphi_{\epsilon}'(r)+t\geq0$. This implies that for each $\epsilon\in(0,\delta_{t})$, $\varphi_{\epsilon}(r)+tr$ is an increasing function of $r$. For each $t>0$, we choose $\epsilon_{t}>0$ so that the function $\epsilon_{t}$ is decreasing and tends to $0$ as $t\to0^{+}$. Moreover, $\varphi_{\epsilon_{t}}(r)+tr$ is an increasing function of $r$. Letting $t\to0^{+}$ and using Proposition \ref{pro2.1}(b), we have
$$\lim_{t\to0^{+}}(\varphi_{\epsilon_{t}}(r)+tr)=\lim_{\epsilon_{t}\to0}\vint_{\partial B_{r}(\boldsymbol{x})}\widetilde{u}^{\epsilon_{t}}|_{\Omega}(\boldsymbol{y})\,dS(\boldsymbol{y})=\vint_{\partial B_{r}(\boldsymbol{x})}
u(\boldsymbol{y})\,dS(\boldsymbol{y})=:\varphi(r).$$
Observe that $\varphi(r)$ is an increasing function of $r$. Hence,
\begin{equation}\label{eq(3.12)}
\varphi(r)\geq\lim_{\xi\to0}\varphi(\xi)=\lim_{\xi\to0}\vint_{\partial B_{\xi}(\boldsymbol{x})}u(\boldsymbol{y})\,dS(\boldsymbol{y})=u(\boldsymbol{x}).
\end{equation}
Using (\ref{eqpj2}) and the equation
$$\int_{B_{r}(\boldsymbol{x})}u(\boldsymbol{y})\,d\boldsymbol{y}=\int_{0}^{r}\Big(\int_{\partial B_{s}(\boldsymbol{x})}u(\boldsymbol{y})\,dS(\boldsymbol{y})\Big)\,ds,$$
which can be derived by using \cite[\S C3]{Evans_2010}, we have
\begin{align}\label{eq3.4a}
\vint_{B_{r}(\boldsymbol{x})}u(\boldsymbol{y})\,d\boldsymbol{y}&=\frac{1}{\alpha(n)r^{n}}\int_{0}^{r}\Big(\int_{\partial B_{s}(\boldsymbol{x})}u(\boldsymbol{y})\,dS(\boldsymbol{y})\Big)\,ds\notag\\
&=\frac{1}{\alpha(n)r^{n}}\int_{0}^{r}n\alpha(n)s^{n-1}\Big(\vint_{\partial B_{s}(\boldsymbol{x})}u(\boldsymbol{y})\,dS(\boldsymbol{y})\Big)\,ds\qquad\quad(\text{by}~(\ref{eqpj1}))\notag\\
&\geq\frac{1}{\alpha(n)r^{n}}\int_{0}^{r}u(\boldsymbol{x})n\alpha(n)s^{n-1}\,ds\qquad\qquad\qquad\qquad\qquad\,\,(\text{by}~(\ref{eq(3.12)}))\notag\\
&=\frac{u(\boldsymbol{x})}{r^{n}}\int_{0}^{r}ns^{n-1}\,ds\notag\\
&=u(\boldsymbol{x}).
\end{align}
Suppose there exists a point $\boldsymbol{x}_{0}\in\Omega$ such that $u(\boldsymbol{x}_{0})=M:=\displaystyle\max_{\Omega}u(\boldsymbol{x})$. Then for $0<r<\text{dist}(\boldsymbol{x}_{0},\partial\Omega)$, we have, by (\ref{eq3.4a})
$$M=u(\boldsymbol{x}_{0})\leq\vint_{B_{r}(\boldsymbol{x}_{0})}u\,d\boldsymbol{y}\leq M.$$
Hence,
$$\vint_{B_{r}(\boldsymbol{x}_{0})}u\,d\boldsymbol{y}=M.$$
Thus $u(\boldsymbol{y})=M$ for all $y\in B_{r}(\boldsymbol{x}_{0})$. Since $\Omega$ is a domain, which is connected, it follows that $u(\boldsymbol{x})=M$ for all $\boldsymbol{x}\in\Omega$.
\end{proof}

\begin{rmk}\label{remark1}
\begin{enumerate}
\item[(a)]
By replacing $u$ in the proof of Theorem \ref{thm1} with $-u$, one can prove that a nonconstant continuous $\mu$-superharmonic function attains its minimum only on $\partial\Omega$.

\item[(b)] From the proof of Theorem \ref{thm1}, one can see that the reason for introducing the function $\widetilde{u}^{\epsilon}|_{\Omega}$ is that $\Delta u$
 may exist as a distribution in $\mathcal{D}(\Omega)$ but need not exist as a function (see Example \ref{EX1}).
\end{enumerate}
\end{rmk}

Let $\Omega$ be a bounded domain. For a continuous {\em $\mu$-harmonic function} $u$, i.e., $\Delta_{\mu}u=0$, we have the following proposition.
\begin{prop}\label{pro0}
Let $\Omega$ be a bounded domain. If $u\in C(\Omega)$ is a $\mu$-harmonic function on $\Omega$ and vanish on $\partial\Omega$, then $u\equiv0$.
\end{prop}
In order to prove Proposition \ref{pro0}, we need the following Weyl's Lemma (see \cite[Corollary 2.2.1]{Jost_2013} or \cite{Weyl_1940,Beltrami_1968,Garding_1951} and references there in).
\begin{lem}[Weyl's Lemma]\label{lem1}
Let $u:~\Omega\to\R$ be measurable and locally integrable on $\Omega$. Suppose that for any $\phi\in C_{c}^{\infty}(\Omega)$
$$\int_{\Omega}u(\boldsymbol{x})\Delta\phi(\boldsymbol{x})\,d\boldsymbol{x}=0.$$
Then $u$ is harmonic and, in particular, smooth.
\end{lem}
\begin{proof}
Omit. (See, e.g., \cite[Corollary 2.2.1]{Jost_2013}).
\end{proof}
\begin{proof}[Proof of Proposition \ref{pro0}]
Let $u$ be $\mu$-harmonic, i.e., $\Delta_{\mu}u=0$. Then for all $v\in C_{c}^{\infty}(\Omega)$
$$0=\int_{\Omega}v\Delta_{\mu}u\,d\mu=-\int_{\Omega}\nabla v\cdot\nabla u\,d\boldsymbol{x}=\int_{\Omega}u \Delta v\,d\boldsymbol{x}.$$
Moreover, $u$ is locally integrable as $u\in{\rm Dom}(\Delta_{\mu})\subset H_{0}^{1}(\Omega)$. It follows from Lemma \ref{lem1} that $u$ is harmonic and smooth on $\Omega$. Therefore, by the classical maximum principle (see e.g., \cite[\S6.4]{Evans_2010}) and the fact that $u$ vanish on $\partial\Omega$, we have $u\equiv0$.\end{proof}

\section{ Courant nodal domain theorem }\label{sec4}

Let $\Omega\subset R^{d}$ be a bounded domain and $\mu$ be a positive finite Borel measure on $\mathbb{R}^{d}$ with ${\rm supp}(\mu)\subset \overline{\Omega}$ and $\mu(\Omega)>0$. Let $\mathcal{E}(u,u)$ be defined as (\ref{(2.1)}). We define the Rayleigh quotient of $\mu$. This is an important and useful quantity
 in studying eigenvalues. For related applications, see \cite{Alessandrini_1998,Gladwell-Zhu_2002,Strauss_2008,Deng-Ngai_2021,Chen-Ngai_2010}.

\begin{defi}
Use the above assumption and notations. For any $u\in\rm{Dom(\mathcal{E})}$, the {\em Rayleigh quotient} associated with $\mu$ is defined as
\begin{equation}
R_{\mu}(u):=\frac{\mathcal{E}(u,u)}{\quad(u,u)_{L^2(\Omega,\mu)}}
=\frac{\displaystyle\int_{\Omega}|\nabla u|^{2} \,d\boldsymbol{x}}{\displaystyle\int_{\Omega}|u|^{2} \,d\mu}.
\end{equation}
\end{defi}

The following lemma is mainly inspired by \cite{Deng-Ngai_2021}. Using a similar method as in the proof of \cite[Theorem 1.3]{Deng-Ngai_2021}, we have the following lemma. We omit the proof.
\begin{lem}\label{lem4.1}
Let $\lambda_{1},\ldots,\lambda_{n}$ be eigenvalues of (\ref{eq(2.2)}) and let $u_{1}(\boldsymbol{x}),\ldots,u_{n}(\boldsymbol{x})$
be corresponding eigenfunctions. Then
\begin{enumerate}
\item[(1)]~$\lambda_{1}={\rm min}\big\{R_{\mu}(u)\,\big|\,u\in {\rm Dom}(\mathcal{E})\big\}$ and the minimizing function is $u_{1}$, i.e., $R_{\mu}(u_{1})=\lambda_{1}${\rm;}

\item[(2)]~$\lambda_{n}={\rm min}\big\{R_{\mu}(u)\,\big|\,u\in {\rm Dom}(\mathcal{E}),\,(u,u_{i})_{L^2(\Omega,\mu)}=0,\,i=1,\ldots,n-1\big\}$ and the minimizing function is $u_{n}$, i.e., $R_{\mu}(u_{n})=\lambda_{n}$.
\end{enumerate}
\end{lem}

\begin{proof}
Omit.
\end{proof}
We now prove Theorem \ref{T1}.
\begin{proof}[Proof of theorem \ref{T1}]
We follow \cite{Strauss_2008} for the proof of (a). We use the methods and techniques in \cite{Gladwell-Zhu_2002} to prove (b).

\noindent(a)~We divide the proof into two steps as follows:

\noindent {\em Step1.} Suppose on the contrary that the $\lambda_{1}$-eigenfunction $u_{1}(\boldsymbol{x})$ has nodes. i.e., there exists $\boldsymbol{x}_{0}\in\Omega$ such that
\begin{equation}\label{TB4}
u_{1}(\boldsymbol{x}_{0})=0.
\end{equation}
Let
$$
\Omega^{+}:=\{\boldsymbol{x}\in\Omega\,|\,u_{1}(\boldsymbol{x})>0\}\qquad\text{and}\qquad
\Omega^{-}:=\{\boldsymbol{x}\in\Omega\,|\, u_{1}(\boldsymbol{x})<0\}.
$$
We claim that $\Omega^{+}$ and $\Omega^{-}$ are nonempty. In fact, if $\Omega^{+}=\emptyset$, then for any $\boldsymbol{x}\in\Omega$,
\begin{equation}\label{TB5}
u_{1}(\boldsymbol{x})\leq0.
\end{equation}
Since $-\Delta_{\mu} u_{1}(\boldsymbol{x})=\lambda_{1} u_{1}(\boldsymbol{x})$, and $\lambda_{1}>0$, we have
$$ -\Delta_{\mu} u_{1}(\boldsymbol{x})=\lambda_{1} u_{1}(\boldsymbol{x})\leq0.$$
Thus
$$\Delta_{\mu} u_{1}(\boldsymbol{x})\geq0,\,  \qquad\text{for~all}~\boldsymbol{x}\in\Omega.$$
Hence $u_{1}(\boldsymbol{x})$ is a $\mu$-subharmonic function. Combining this with Theorem \ref{thm1}, we see that $u_{1}(\boldsymbol{x})$ cannot attain $0$ in $\Omega$. Thus, by (\ref{TB5}),
\begin{equation*}
u_{1}(\boldsymbol{x})<0,\,\qquad \text{for~all}~\boldsymbol{x}\in\Omega,
\end{equation*}
which contradicts (\ref{TB4}). Hence~$\Omega^{+}\neq\emptyset$. By the same argument, $\Omega^{-}\neq\emptyset$.

\noindent {\em Step 2.} Let
$$
u^{+}(\boldsymbol{x})=\begin{cases}
u_{1}(\boldsymbol{x}),\,\qquad \boldsymbol{x}\in\Omega^{+},\\
0,\qquad \qquad \boldsymbol{x}\in\Omega^{-},
\end{cases}
$$
and let $u^{-}(\boldsymbol{x})=u_{1}(\boldsymbol{x})-u^{+}(\boldsymbol{x})$.
Note that $|u_{1}(\boldsymbol{x})|=u^{+}(\boldsymbol{x})-u^{-}(\boldsymbol{x})$.
Since for any $\boldsymbol{x}\in\partial\Omega$, $u_{1}(\boldsymbol{x})=0$, we have
$$
\nabla u^{+}(\boldsymbol{x})=\begin{cases}
\nabla u_{1}(\boldsymbol{x}),\quad&\boldsymbol{x}\in\Omega^{+},\\
0,&\boldsymbol{x}\in\Omega\backslash\Omega^{+},
\end{cases}
\qquad\text{and}\qquad
\nabla u^{-} (\boldsymbol{x})
=\begin{cases}
0,\quad&\boldsymbol{x}\in\Omega\backslash\Omega^{-},\\
\nabla u_{1}(\boldsymbol{x}) &\boldsymbol{x}\in\Omega^{-}.
\end{cases}
$$
Hence,
\begin{align*}
R_{\mu}\big(|u_{1}|\big)
&=\frac{\displaystyle\int_{\Omega}\big|\nabla|u_{1}|\big|^{2} \,d\boldsymbol{x}}{\displaystyle\int_{\Omega}|u_{1}|^{2} \,d\mu}=\frac{\displaystyle\int_{\Omega}|\nabla  u_{1}|^{2} \,d\boldsymbol{x}}{\displaystyle\int_{\Omega}|u_{1}|^{2} \,d\mu}=\frac{\displaystyle\int_{\Omega}(-\Delta_{\mu} u_{1})\cdot u_{1} \,d\mu}{\displaystyle\int_{\Omega}|u_{1}|^{2} \,d\mu}\\
&=\frac{\lambda_{1}\displaystyle\int_{\Omega}u_{1}^{2} \,d\mu}{\displaystyle\int_{\Omega}|u_{1}|^{2} \,d\mu}=\lambda_{1}.
\end{align*}
According to Lemma \ref{lem4.1}, $|u_{1}(\boldsymbol{x})|$ is a $\lambda_{1}$-eigenfunction, i.e.,
\begin{equation}\label{TB6}
-\Delta_{\mu} |u_{1}(\boldsymbol{x})|=\lambda_{1} |u_{1}(\boldsymbol{x})|.
\end{equation}
Combining this with $\lambda_{1}>0$, we have
$$\Delta_{\mu} |u_{1}(\boldsymbol{x})|\leq0,  \qquad  \text{for all}~\boldsymbol{x}\in\Omega.$$
Hence $|u_{1}(\boldsymbol{x})|$ is a $\mu$-superharmonic function. By Remark \ref{remark1}, $u_{1}(\boldsymbol{x})$ cannot
attain $0$ in $\Omega$. Thus,
$$
|u_{1}(\boldsymbol{x})|>0,\quad \text{for~all}~\boldsymbol{x}\in\Omega,
$$
i.e., $|u_{1}(\boldsymbol{x})|$ does not have nodes in $\Omega$. This contradicts Step 1 and completes the first part of the theorem.

\noindent(b) By the proof of (a),  $u_{1}(\boldsymbol{x})\neq0$, for any $\boldsymbol{x}\in\Omega$. Without loss of generality, we assume $u_{1}(\boldsymbol{x})>0$. Since $u_{n}(\boldsymbol{x})$ is orthogonal to $u_{1}(\boldsymbol{x})$, i.e.,
$$\displaystyle\int_{\Omega}u_{n}(\boldsymbol{x})u_{1}(\boldsymbol{x}) \,d\mu=0,$$
we have, $u_{n}(\boldsymbol{x})$ must change sign in $\Omega$. Thus $u_{n}(\boldsymbol{x})$ must be positive in some subdomians
of $\Omega$ and negative in some subdomains of $\Omega$. By  the continuity of $u_{n}(\boldsymbol{x})$, these subdomains must be separated by the
nodal set of $u_{n}(\boldsymbol{x})$. Hence, for $n\geq 2$, the nodal set of $u_{n}(\boldsymbol{x})$ divides $\Omega$ into at least two subdomains.

For the second part of (b), let $\mathcal{Z}_{\mu}$ be defined as in (\ref{defiZ}). Then $\Omega\setminus\mathcal{Z}_{\mu}(u_{n})=\{\boldsymbol{x}\in\Omega \,|\, u_{n}(\boldsymbol{x})\neq0\}$.
Assume $\mathcal{Z}_{\mu}(u_{n})$ divides $\Omega$ into $m$ subdomains: $\Omega_{1},\ldots,\Omega_{m}$, every two $\Omega_{i}$ are disjoint and separated by the subset of $\mathcal{Z}_{\mu}(u_{n})$ and
$$\Omega\setminus\mathcal{Z}_{\mu}(u_{n})=\bigcup_{j=1}^{m}\Omega_{j}.$$
Let
$$
w_{j}(\boldsymbol{x})=\begin{cases}
u_{n}(\boldsymbol{x}),\,\quad\qquad \boldsymbol{x}\in\Omega_{j},\\
0,\,\qquad \qquad\quad \boldsymbol{x}\in\Omega\backslash\Omega_{j}.
\end{cases}
$$
Then
$$
\nabla w_{j}(\boldsymbol{x})=\begin{cases}
\nabla u_{n}(\boldsymbol{x}),\,\qquad \boldsymbol{x}\in\Omega_{j},\\
0,\qquad \qquad\quad \boldsymbol{x}\in\Omega\backslash\Omega_{j}.
\end{cases}
$$
Let
\begin{equation}\label{TB8}
w(\boldsymbol{x})=\sum_{j=1}^{m}c_{j}w_{j}(\boldsymbol{x}),
\end{equation}
where $c_{1},\ldots,c_{m}$ are arbitrary constants. We claim that the Rayleigh quotient of $w(\boldsymbol{x})$ is $\lambda_{n}$. In fact,
\begin{align*}
R_{\mu}(w)&=\frac{\displaystyle\int_{\Omega}|\nabla w|^{2} \,d\boldsymbol{x}}{\displaystyle\int_{\Omega}|w|^{2} \,d\mu}
=\frac{\sum_{j=1}^{m}c_{j}^{2}\displaystyle\int_{\Omega_{j}}|\nabla u_{n}|^{2} \,d\boldsymbol{x}}{\sum_{j=1}^{m}c_{j}^{2}\displaystyle\int_{\Omega_{j}}|u_{n}|^{2} \,d\mu}=\frac{\sum_{j=1}^{m}c_{j}^{2}\displaystyle\int_{\Omega_{j}}(-\Delta_{\mu} u_{n})\cdot u_{n} \,d\mu}{\sum_{j=1}^{m}c_{j}^{2}\displaystyle\int_{\Omega_{j}}|u_{n}|^{2} \,d\mu}\\
&=\frac{\lambda_{n}\sum_{j=1}^{m}c_{j}^{2}\displaystyle\int_{\Omega_{j}}u_{n}^{2} \,d\mu}{\sum_{j=1}^{m}c_{j}^{2}\displaystyle\int_{\Omega_{j}}|u_{n}|^{2} \,d\mu}=\lambda_{n}.
\end{align*}
We can choose the coefficients $\{c_{j}\}_{j=1}^{m}$ of $w(\boldsymbol{x})$ in (\ref{TB8}) such that
$$
\big(w(\boldsymbol{x}),u_{i}(\boldsymbol{x})\big)_{L^2(\Omega,\mu)}=0,\,\qquad i=1,\ldots,m-1,
$$
where $\{u_{i}\}_{i=1}^{m-1}$ is the first $m-1$ eigenfunctions. Hence, for this choice of $\{c_{j}\}_{j=1}^{m}$, by Lemma \ref{lem4.1}, we have
$$R_{\mu}(w)\geq\lambda_{m},$$
i.e., $\lambda_{m}\leq\lambda_{n}$. Since $\lambda_{n}<\lambda_{n+r}$, we have
\begin{equation*}
\lambda_{m}<\lambda_{n+r}.
\end{equation*}
Thus
\begin{equation*}
 m\leq n+r-1.
\end{equation*}
Therefore, the nodal set of $u_{n}(\boldsymbol{x})$ divides $\Omega$ into at most $n+r-1$ subdomains.
\end{proof}
According to Theorem {\rm\ref{T1}}(a), we can immediately derive the following corollary.
\begin{coro}\label{C1}
The multiplicity of the first eigenvalue $\lambda_{1}$ is $1$, i.e., $\lambda_{1}$ is a simple eigenvalue.
\end{coro}

\begin{proof}[\textbf{Proof}]
Assume on the contrary that the multiplicity of $\lambda_{1}$ is not $1$. Thus there exists another $\lambda_{1}$-eigenfunction $\eta(\boldsymbol{x})$ so that $\{u_{1}(\boldsymbol{x}), \eta(\boldsymbol{x})\}$ is linearly independent.  Let $E_{1}$ be the eigenspace corresponding to $\lambda_{1}$. We have
$$E_{1}=\big\langle u_{1}(\boldsymbol{x})\big\rangle\oplus\big\langle u_{1}(\boldsymbol{x})\big\rangle^{\bot},$$
where $\big\langle u_{1}(\boldsymbol{x})\big\rangle$ is the subspace generated by the eigenfunction $u_{1}(\boldsymbol{x})$, and $\big\langle u_{1}(\boldsymbol{x})\big\rangle^{\bot}$ is the orthogonal complement of $\big\langle u_{1}(\boldsymbol{x})\big\rangle$ in $E_{1}$. We have
$$\eta(\boldsymbol{x})=c_{1}u_{1}(\boldsymbol{x})+w(\boldsymbol{x}),$$
where $c_{1}$ is a constant and $w(\boldsymbol{x})\in\big\langle u_{1}(\boldsymbol{x})\big\rangle^{\bot}\subset E_{1}$, i.e.,
\begin{equation}\label{TB9}
\displaystyle\int_{\Omega}u_{1}(\boldsymbol{x})w(\boldsymbol{x}) \,d\mu=0.
\end{equation}
According to the proof of Theorem {\rm\ref{T1}}, we deduce that $w(\boldsymbol{x})$ does not change sign in $\Omega$. Combining this and (\ref{TB9}), we have $w(\boldsymbol{x})\equiv0$ in $\Omega$. Thus
$$\eta(\boldsymbol{x})=c_{1}u_{1}(\boldsymbol{x}),\,\qquad \boldsymbol{x}\in\Omega,$$
contradicting the linear independence of $u_{1}(\boldsymbol{x})$ and $\eta(\boldsymbol{x})$. This complete the proof.
\end{proof}
\begin{rmk}\label{rem4.1}
One can see that $E_{1}$ is a one-dimensional eigenspace.
\end{rmk}

\section{  Continuity of eigenfunctions }\label{sec5}

In this section, we prove Theorem \ref{thm2}. Let $\Omega\subset\R^{d}$ be a bounded domain. We first recall the Lipschitz function defined on $\Omega$ (see, e.g.,\cite{Medkova_2018}).

Let $C^{0,1}(\Omega)$ be the set of all bounded continuous functions on $\Omega$ equipped with the norm
$$\|f\|_{C^{0,1}(\Omega)}:=\sup_{\boldsymbol{x}\in\Omega}|f(\boldsymbol{x})|+\sup_{\boldsymbol{x},\boldsymbol{y}\in\Omega,\boldsymbol{x}\neq \boldsymbol{y}}\frac{|f(\boldsymbol{x})-f(\boldsymbol{y})|}{|\boldsymbol{x}-\boldsymbol{y}|}.$$
Any function $f\in C^{0,1}(\Omega)$ is called a {\em Lipschitz function} on $\Omega$. Let $d\geq 2$ and $f$ be a Lipschitz function defined on $\R^{d-1}$. We say that $\Omega\subset\R^{d}$ is a {\em Lipschitz graph domain} corresponding to $f$ if
$$\Omega=\{(\boldsymbol{x},x):\boldsymbol{x}\in\R^{d-1},~x\in\R,~\boldsymbol{x}\geq x\}.$$
\begin{defi}\label{defibdary}
Let $\Omega\subset\R^{d}$ be a bounded domain. We say that $\Omega$ has {\em Lipschitz boundary} if for any $\boldsymbol{x}\in\partial\Omega$ there exists a ball $B_{r}(\boldsymbol{x})$ with $r>0$ and a Lipschitz graph domain $D$ such that $B_{r}(\boldsymbol{x})\cap\Omega=B_{r}(\boldsymbol{x})\cap D$.
\end{defi}
Let $\Omega\subset\R^{d}$ be a bounded domain. We recall the Green function of the classical Laplacian $\Delta$. For $u\in C^{2}(\Omega)$,
$$\Delta u=\sum_{i=1}^{d}\frac{\partial^{2}u}{\partial x^{2}_{i}}.$$
Let
\begin{equation}\label{eqg}
g(\boldsymbol{x},\boldsymbol{y})=\begin{cases}-\frac{1}{2\pi}\ln|\boldsymbol{x}-\boldsymbol{y}|\quad&\text{if}~d=2,\\
-|\boldsymbol{x}-\boldsymbol{y}|^{2-d} &\text{if}~d\geq3.
\end{cases}
\end{equation}
In the case $d\geq2$, the Green function is defined as
\begin{equation}\label{eq(5.20)}
G(\boldsymbol{x},\boldsymbol{y}):=g(\boldsymbol{x},\boldsymbol{y})+h(\boldsymbol{x},\boldsymbol{y}),
\end{equation}
where $h(\boldsymbol{x},\boldsymbol{y})$ is a symmetric continuous function on $\overline{\Omega}\times\overline{\Omega}$ and is harmonic
 in $\boldsymbol{x}\in\Omega$ for arbitrary $\boldsymbol{y}\in\Omega$. Some basic properties of the Green function are summarized below (see \cite{Doob_2001,Medkova_2018}), if we fix any point $\boldsymbol{x}\in\Omega$
\noindent\begin{enumerate}
\item[(a)] $G(\boldsymbol{x},\cdot)$ is defined on $\Omega\times\Omega$ and $G(\boldsymbol{x},\boldsymbol{x})=+\infty$.

\item[(b)] $G(\boldsymbol{x},\cdot)$ is harmonic on $\Omega-\{\boldsymbol{x}\}$.

\item[(c)] If additionally, $\Omega$ has Lipschitz boundary, then $G(\boldsymbol{x},\cdot)$ has limit $0$ at any point of $\partial\Omega$ (see, e.g.,\cite[Theorem 4.6.2]{Medkova_2018}).
\end{enumerate}
Also, it is known that $G(\boldsymbol{x},\boldsymbol{y})$ is symmetric and continuous on $\Omega\times\Omega$ \cite{Doob_2001}.
Note that if $\Omega$ is bounded domain of $\R^{d}$ $(d\geq2)$ with Lipschitz boundary, then the symmetry of $G(\boldsymbol{x},\boldsymbol{y})$ and (c) imply
\begin{equation}\label{eq(3.1)}
\lim_{\boldsymbol{x}\to \boldsymbol{z}}G(\boldsymbol{x},\boldsymbol{y})=0, \quad\text{for~any}~\boldsymbol{z}\in\partial\Omega~\text{and}~\boldsymbol{y}\in\Omega.
\end{equation}

\begin{rmk}\label{Rek5.2}
For any open set $\Omega\subset\R^{d}$, the Green function always exists when $d\geq3$. When $d=2$, the Green function exists if $\R^{2}\setminus\partial\Omega$ is not connected (see \cite[Theorm 4.1.2]{Armitage-Gardiner_2001} also see \cite{Medkova_2018,Bass_1995}).
\end{rmk}
For $f\in C^{1}(\Omega)$, the unique solution of Possion equation in $C^{2}(\Omega)$:
$$\begin{cases}-\Delta u=f\\
u|_{\partial\Omega}=0
\end{cases}$$
can be represented through the Green function $G(\boldsymbol{x},\boldsymbol{y})$ by
$$u(\boldsymbol{x})=\int_{\Omega}G(\boldsymbol{x},\boldsymbol{y})f(\boldsymbol{y})\,d\boldsymbol{y}.$$
More details about the Green function can be found in, for example, \cite{Armitage-Gardiner_2001,Evans_2010, Medkova_2018, Courant-Hilbert_1953, Courant-Hilbert_1989, Doob_2001}.

By \cite{Hu-Lau-Ngai_2006}, the Green function $G(\boldsymbol{x},\boldsymbol{y})$ for $\Delta$,
if exists, is also the Green function for $\Delta_{\mu}$. It means that for the equation
$$-\Delta_{\mu}u=f,$$
there exists a Green operator defined on $L^{p}(\Omega,\mu)$ by
\begin{equation}\label{eq(5.21)}
(G_{\mu}f)(\boldsymbol{x}):=\int_{\Omega}G(\boldsymbol{x},\boldsymbol{y})f(\boldsymbol{y})\,d\mu(\boldsymbol{y})
\end{equation}
such that $u=G_{\mu}f$. The operator $G_{\mu}$ is the inverse of $-\Delta_{\mu}$ \cite[Theorem 1.3]{Hu-Lau-Ngai_2006}. To ensure that $G_{\mu}$ has good properties, we need the following assumption in \cite{Hu-Lau-Ngai_2006}:
\begin{equation}\label{eqc2}
\sup_{\boldsymbol{x}\in\Omega}\int_\Omega G(\boldsymbol{x},\boldsymbol{y})\,d\mu(\boldsymbol{y})\leq C<+\infty\quad\text{for~some~constant}~C>0.
\end{equation}
It is proved in \cite{Hu-Lau-Ngai_2006} that the condition $\underline{\dim}_{\infty}(\mu)>d-2$ implies (\ref{eqc2}).
\begin{prop}\label{pro5.1}
Let $\Omega\subset\R^{d}$ be a bounded domain on which the classical Green function $G(\boldsymbol{x},\boldsymbol{y})$ exists and let $f\in {\rm Dom}(\Delta_{\mu})$. Let $\mu$ be a finite positive Borel measure with {\rm supp}$(\mu)\subset\overline\Omega$. Assume $\underline{\dim}_{\infty}(\mu)>d-2$. Then $G_{\mu}f$ is bounded, i.e., there exists a constant $\widetilde{C}>0$ such that $|G_{\mu}f(\boldsymbol{x})|\leq \widetilde{C}$ for all $\boldsymbol{x}\in\Omega$.
\end{prop}
\begin{proof}
Since $d=1$ is clear, we divide the proof into two cases: $d=2$ and $d\geq3$.

\noindent{\em Case 1. $d=2$.} We claim that for each $\boldsymbol{x}\in\Omega$, $G(\boldsymbol{x},\boldsymbol{y})\in L^{2}(\Omega,\mu)$. By (\ref{eq(5.20)}), it suffices to prove that there exists some constant $\widetilde{C}>0$ such that.
\begin{equation}\label{eq(5.25)}
\int_{\Omega}(\ln{|\boldsymbol{x}-\boldsymbol{y}|})^{2}\,d\mu(\boldsymbol{y})\leq\widetilde{C}
\end{equation}
for all $\boldsymbol{x}\in\Omega$. Using the same method in the proof of \cite[Proposition 4.1]{Hu-Lau-Ngai_2006}, one can prove that (\ref{eq(5.25)}) holds. Hence,
\begin{equation}\label{eq(5.26)}
\big|G_{\mu}f(\boldsymbol{x})\big|=\Big|\int_{\Omega}G(\boldsymbol{x},\boldsymbol{y})f(\boldsymbol{y})\,d\mu(\boldsymbol{y})\Big|\leq\|G(\boldsymbol{x},\cdot)\|_{L^{2}(\Omega,\mu)}\|f\|_{L^{2}(\Omega,\mu)}.
\end{equation}

\noindent{\em Case 2. $d\geq3$.} Since $f\in{\rm Dom}(\Delta_{\mu})\subset H_{0}^{1}(\Omega)$ and $\underline{\dim}_{\infty}(\mu)>d-2$, there exists a sequence $f_{m}\in C_{c}^{\infty}(\Omega)$ such that $f_{m}\to f$ in $L^{2}(\Omega,\mu)$. We claim that there exists some constant $C_{1}>0$ such that
\begin{equation}\label{eq(5.27)}
\lim_{m\to\infty}\big|G_{\mu}(f^{2}-f_{m}^{2})\big|\leq C_{1}.
\end{equation}
To see this, by (\ref{eq(5.20)}), it suffices to prove
\begin{equation}\label{eqgg}
\lim_{m\to\infty}\int_{\Omega}|g(\boldsymbol{x},\boldsymbol{y})|\cdot|f^{2}(\boldsymbol{y})-f_{m}^{2}(\boldsymbol{y})|\,d\mu(\boldsymbol{y})\leq C_{1}.
\end{equation}
Note that, by H${\rm\ddot{o}}$lder's inequality,
\begin{align}\label{eq(5.28)}
\int_{\Omega}|f^{2}-f_{m}^{2}|\,d\mu=&\int_{\Omega}|f-f_{m}||f+f_{m}|\,d\mu\notag\\
\leq&\|f-f_{m}\|_{L^{2}(\Omega,\mu)}\|f+f_{m}\|_{L^{2}(\Omega,\mu)}\to0
\end{align}
as $m\to\infty$.
Let ${\rm diam(\Omega)}=r_{0}$. We have
\begin{align*}\int_{\Omega}|g(\boldsymbol{x}-\boldsymbol{y})||f^{2}(\boldsymbol{y})-f_{m}^{2}(\boldsymbol{y})|\,d\mu(\boldsymbol{y})=&\int_{|\boldsymbol{x}-\boldsymbol{y}|<1}|\boldsymbol{x}-\boldsymbol{y}|^{-(d-2)}|f^{2}(\boldsymbol{y})-f_{m}^{2}(\boldsymbol{y})|\,d\mu(\boldsymbol{y})\\
&+\int_{1\leq|\boldsymbol{x}-\boldsymbol{y}|\leq r_{0}}|\boldsymbol{x}-\boldsymbol{y}|^{-(d-2)}|f^{2}(\boldsymbol{y})-f_{m}^{2}(\boldsymbol{y})|\,d\mu(\boldsymbol{y}).
\end{align*}
By (\ref{eq(5.28)}), the second integral on the right-hand side tends to $0$ as
\begin{equation}\label{eq(5.29)}
\int_{1\leq|\boldsymbol{x}-\boldsymbol{y}|\leq r_{0}}|\boldsymbol{x}-\boldsymbol{y}|^{-(d-2)}|f^{2}(\boldsymbol{y})-f_{m}^{2}(\boldsymbol{y})|\,d\mu(\boldsymbol{y})\leq\int_{1\leq|\boldsymbol{x}-\boldsymbol{y}|\leq r_{0}}|f^{2}(\boldsymbol{y})-f_{m}^{2}(\boldsymbol{y})|\,d\mu(\boldsymbol{y}).
\end{equation}
Furthermore, we let $B_{k}(\boldsymbol{x}):=\{\boldsymbol{y}|~2^{-k}\leq|\boldsymbol{x}-\boldsymbol{y}|\leq2^{-(k-1)}\}$. Then
\begin{align*}
&\int_{|\boldsymbol{x}-\boldsymbol{y}|\leq 1}|\boldsymbol{x}-\boldsymbol{y}|^{-(d-2)}|f^{2}(\boldsymbol{y})-f_{m}^{2}(\boldsymbol{y})|\,d\mu(\boldsymbol{y})\\
&=\sum_{k=1}^{\infty}\int_{B_{k}(\boldsymbol{x})}|\boldsymbol{x}-\boldsymbol{y}|^{-(d-2)}|f^{2}(\boldsymbol{y})-f^{2}_{m}(\boldsymbol{y})|\,d\mu(\boldsymbol{y})\\
&\leq\lim_{N\to\infty}\sum_{k=1}^{N}2^{k(d-2)}\int_{B_{k}(x)}|f^{2}(\boldsymbol{y})-f^{2}_{m}(\boldsymbol{y})|\,d\mu(\boldsymbol{y}).
\end{align*}
By (\ref{eq(5.28)}), for each $k\in\{1,2,\ldots,N\}$, there exists $m_{N}$ sufficiently large such that
$$\int_{B_{k}(\boldsymbol{x})}|f^{2}(\boldsymbol{y})-f_{m_{N}}^{2}(\boldsymbol{y})|\,d\mu(\boldsymbol{y})\leq2^{-2k(d-2)}.$$
Hence
$$\sum_{k=1}^{N}2^{k(d-2)}\int_{B_{k}(\boldsymbol{x})}|f^{2}(\boldsymbol{y})-f_{m_{N}}^{2}(\boldsymbol{y})|\,d\mu(\boldsymbol{y})\leq\sum_{k=1}^{N}2^{-k}.$$
Letting $N\to\infty$, we have
$$\int_{|\boldsymbol{x}-\boldsymbol{y}|<1}|\boldsymbol{x}-\boldsymbol{y}|^{-(d-2)}|f^{2}(\boldsymbol{y})-f_{m}^{2}(\boldsymbol{y})|\,d\mu(\boldsymbol{y})\leq\sum_{k=1}^{\infty}2^{-k}<+\infty.$$
Combining this and (\ref{eq(5.29)}) completes the proof of the claim. Note that the claim implies that
\begin{equation}\label{eq(5.30)}
\big|G_{\mu}f^{2}(\boldsymbol{x})\big|\leq C_{3},
\end{equation}
for some constant $C_{3}>0$. In fact, by (\ref{eq(5.27)}), there exists integer $N_{0}$ sufficiently large such that for all $m>N_{0}$ and all $\boldsymbol{x}\in\Omega$
\begin{equation}\label{eqg1}
\big|G_{\mu}(f^{2}(\boldsymbol{x})-f_{m}^{2}(\boldsymbol{x}))\big|\leq C_{1}+1.
\end{equation}
In particular,
\begin{align}\label{eq(5.31)}
\big|G_{\mu}f^{2}\big|&\leq\big|G_{\mu}f_{N_{0}+1}^{2}\big|+C_{2}\notag\\
&\leq\int_{\Omega}\big|G(\boldsymbol{x},\boldsymbol{y})f^{2}_{N_{0}+1}\big|\,d\mu+C_{2}\notag\qquad\qquad\qquad\qquad\,\,(\text{by}~(\ref{eq(5.21)}))\\
&\leq\|f^{2}_{N_{0}+1}\|_{L^{\infty}(\Omega)}\int_{\Omega}G(\boldsymbol{x},\boldsymbol{y})\,d\mu(y)+C_{2}\notag\qquad\qquad(\text{by~H${\rm\ddot{o}}$lder's inequality})\\
&\leq C\|f^{2}_{N_{0}+1}\|_{L^{\infty}(\Omega)}+C_{2}\notag\qquad\qquad\qquad\qquad\qquad\,\,\,\,(\text{by}~(\ref{eq(5.30)}))\\
&=:C_{3},
\end{align}
where $C_{2}=C_{1}+1$.
Now, for $d\geq3$, by H\"{o}lder's inequality, (\ref{eqc2}) and (\ref{eq(5.30)}), we have
\begin{align*}
\big|G_{\mu}f(\boldsymbol{x})\big|^{2}&=\left|\int_{\Omega}G(\boldsymbol{x},\boldsymbol{y})f(\boldsymbol{y})\,d\mu(\boldsymbol{y})\right|^{2}\\
&\leq\left(\int_{\Omega}\big|G(\boldsymbol{x},\boldsymbol{y})^{1/2}f(\boldsymbol{y})\cdot G(\boldsymbol{x},\boldsymbol{y})^{1/2}\big|\,d\mu(\boldsymbol{y})\right)^{2}\\
&\leq\int_{\Omega}G(\boldsymbol{x},\boldsymbol{y})f^{2}(\boldsymbol{y})\,d\mu(\boldsymbol{y})\cdot\int_{\Omega}G(\boldsymbol{x},\boldsymbol{y})\,d\mu(\boldsymbol{y})\\
&\leq C\cdot C_{3}.\qquad\qquad\qquad\qquad\qquad\qquad\qquad(\text{by}~(\ref{eq(5.31)}))
\end{align*}
Therefore $\big|G_{\mu}f\big|$ is bounded by $\sqrt{C\cdot C_{3}}$. Combining this and (\ref{eq(5.26)}) completes the proof.
\end{proof}

The following proposition and the related proof use idea from \cite[Theorem 4.6.3]{Medkova_2018}.
\begin{prop}\label{pro3.1}
Assume the same hypothesis on Proposition \ref{pro5.1} and additionally assume that $\Omega$ has Lipschitz boundary. Let $g$ be defined as in (\ref{eqg}). Let $\boldsymbol{z}\in\partial\Omega$. Then for $\boldsymbol{x}\in\Omega$,
$$\lim_{\boldsymbol{x}\to \boldsymbol{z}}(G_{\mu}f)(\boldsymbol{x})=0.$$
\end{prop}
\begin{proof}
By (\ref{eq(5.25)}), we have, in the case $d=2$, $g(\boldsymbol{x},\boldsymbol{y})+C\in L^{2}(\Omega,\mu)$ for each $\boldsymbol{x}\in\Omega$. Therefore,
for any fixed $\epsilon>0$, there exists $r_{1}>0$ sufficiently small such that
\begin{equation}\label{eq+1}
\int_{B_{r_{1}}(\boldsymbol{z})}|g(\boldsymbol{x},\boldsymbol{y})+C|^{2}\,d\mu(\boldsymbol{y})\leq\epsilon^{2}.
\end{equation}
In view of (\ref{eqc2}), in the case $d\geq3$, there exists $r_{2}>0$ sufficiently small such that
\begin{equation}\label{eq+2}
\int_{B_{r_{2}}(\boldsymbol{z})}G(\boldsymbol{x},\boldsymbol{y})\,d\mu(\boldsymbol{y})<\epsilon^{2}.
\end{equation}
Let $r=\min\{r_{1},r_{2}\}$ and let $f_{1}:=f\chi_{B_{r/2}(\boldsymbol{z})}$ and $f_{2}=f-f_{1}$. For case $d=2$, by H\"{o}lder's inequality and (\ref{eq+1}), we have
\begin{align}\label{eq5.15}
|G_{\mu}f_{1}(\boldsymbol{x})|&=\Big|\int_{\Omega}G(\boldsymbol{x},\boldsymbol{y})f_{1}(\boldsymbol{y})\,d\mu(\boldsymbol{y})\Big|\notag\\
&\leq\|f_{1}\|_{L^{2}(\Omega,\mu)}\Big(\int_{B_{r}(\boldsymbol{z})}|g(\boldsymbol{x},\boldsymbol{y})+C|^{2}\,d\mu(\boldsymbol{y})\Big)^{1/2}\notag\\
&\leq\epsilon\|f\|_{L^{2}(\Omega,\mu)}.
\end{align}
For the case $d\geq3$, by H\"{o}lder's inequality and (\ref{eq+2}), we have
\begin{align*}
\big|G_{\mu}f_{1}(\boldsymbol{x})\big|^{2}&=\left|\int_{\Omega}G(\boldsymbol{x},\boldsymbol{y})f_{1}(\boldsymbol{y})\,d\mu(\boldsymbol{y})\right|^{2}\\
&\leq\int_{B_{r}(\boldsymbol{z})}G(\mathbf{\boldsymbol{x}},\boldsymbol{y})f_{1}^{2}(\boldsymbol{y})\,d\mu(\boldsymbol{y})\cdot\int_{B_{r}(\boldsymbol{z})}G(\boldsymbol{x},\boldsymbol{y})\,d\mu(\boldsymbol{y})\\
&\leq\epsilon^{2}\int_{\Omega}G(\boldsymbol{x},\boldsymbol{y})f^{2}\,d\mu(\boldsymbol{y}).
\end{align*}
Combining this and (\ref{eq(5.30)}), we have
\begin{equation}\label{eq(5.33)}
\big|G_{\mu}f_{1}(\boldsymbol{x})\big|\leq\sqrt{C_{3}}\cdot\epsilon.
\end{equation}
By (\ref{eq(3.1)}), we have
\begin{equation}\label{eq3.1}
\lim_{\boldsymbol{x}\to \boldsymbol{z}}G(\boldsymbol{x},\boldsymbol{y})=0,\qquad \boldsymbol{y}\in\Omega.
\end{equation}
Moreover, there exists $\delta\in(0,r/4)$ such that for any $\boldsymbol{x}\in B_{\delta}(z)$ and all $\boldsymbol{y}\in\Omega\setminus B_{r/2}(\boldsymbol{z})$,
$$G(\boldsymbol{x},\boldsymbol{y})\leq \widetilde{C}:=\max\left\{\big|\log{\frac{r}{4}}\big|,\Big(\frac{r}{4}\Big)^{2-d}\right\}.$$
Hence, for all $\boldsymbol{y}\in\Omega\setminus B_{r/2}(\boldsymbol{z})$ and $\boldsymbol{x}\in B_{\delta}(\boldsymbol{z})$   
\begin{equation}\label{eq3.2}
\big|G(\boldsymbol{x},\boldsymbol{y})f_{2}(\boldsymbol{y})\big|\leq\widetilde{C}\big|f(\boldsymbol{y})\big|.
\end{equation}
Combining (\ref{eq3.1}) and (\ref{eq3.2}), we can use the Lebesgue dominated convergence theorem (see e.g.,\cite{Royden_1988}) to obtain
$$\lim_{\boldsymbol{x}\to \boldsymbol{z}}(G_{\mu}f_{2})(\boldsymbol{x})=0.$$
Therefore, there is $0<\widetilde{r}<r$ such that
\begin{equation}\label{eq3.2a}
\big|(G_{\mu}f_{2})(\boldsymbol{x})\big|\leq\epsilon\qquad\text{for}~\boldsymbol{x}\in B_{\widetilde{r}}(\boldsymbol{z})\cap\Omega.
\end{equation}
Hence, by (\ref{eq5.15}), (\ref{eq(5.33)}) and (\ref{eq3.2a}), we have
$$\big|(G_{\mu}f)(\boldsymbol{x})\big|\leq\big|(G_{\mu}f_{1})(\boldsymbol{x})\big|+\big|(G_{\mu}f_{2})(\boldsymbol{x})\big|\leq\big(\|f\|_{L^{2}(\Omega,\mu)}+\sqrt{C_{3}}+1\big)\cdot\epsilon.$$
\end{proof}

 We give an outline of the proof of Theorem \ref{thm2}. First, we use the continuity of Green's function and the Lebesgue dominated convergence theorem to obtain the continuity of eigenfunction $u$ on the interior of $\Omega$. Second, in the case that $\Omega$ has Lipschitz boundary, to prove that the continuity of $u$ can be extended to $\partial\Omega$, we use the property that $G(\boldsymbol{x},\boldsymbol{y})$ tends to zero when $\boldsymbol{x}$ tends to $\partial\Omega$, which allows us to prove that $u(\boldsymbol{x})$ tends to its boundary value $u|_{\partial\Omega}=0$ as $\boldsymbol{x}$ tends to $\partial\Omega$. Finally, we use condition (\ref{eqc2}) and the dominated convergence theorem to prove that $G_{\mu}u$ is continuous on sufficiently small $r$-balls covering $\partial\Omega$.

\begin{proof}[Proof of Theorem \ref{thm2}]
We divide the proof into two steps.

\noindent {\em Step 1.} Let $u\in{\rm Dom}(\Delta_{\mu})$ be a $\lambda$-eigenfunction of $\Delta_{\mu}$.  Since $-\Delta_{\mu}(G_{\mu}u)=u$ (see \cite[Section 4]{Hu-Lau-Ngai_2006}),
we conclude that each $\lambda$-eigenfunction of $\Delta_{\mu}$ can be expressed as
$$u=-\lambda G_{\mu}u$$
with $u=0$ on $\partial\Omega$. Hence, to prove Theorem \ref{thm2}, it suffices to prove that $G_{\mu}u(\boldsymbol{x})$ is continuous on $\Omega$ for $u\in{\rm Dom}(\Delta_{\mu})$. Since $G(\boldsymbol{x},\boldsymbol{y})$ is continuous on $\Omega\times\Omega$, it is continuous at $\boldsymbol{x}\in\Omega$ for any $\boldsymbol{y}\in\Omega$. Let $r_{1}$ and $r_{2}$ be defined as in \ref{eq+1} and \ref{eq+2}. It is similar to the proof of Proposition \ref{pro3.1}, for $f\in L^{2}(\Omega,\mu)$ and any $\boldsymbol{z}\in\Omega$, let $r=\min\{r_{1},r_{2}\}$ and let $f_{1}:=f\chi_{B_{r/2}(\boldsymbol{z})}$ and $f_{2}=f-f_{1}$. We claim that $G_{\mu}f_{2}$ is continuous at $\boldsymbol{z}$. In fact, for any $y\in\Omega\setminus B_{r/2}(\boldsymbol{z})$, $G(\boldsymbol{x},\boldsymbol{y})$ is continuous at $\boldsymbol{z}$, i.e.,
\begin{equation}\label{eq5.15aa}
\lim_{\boldsymbol{x}\to \boldsymbol{z}}G(\boldsymbol{x},\boldsymbol{y})=G(\boldsymbol{z},\boldsymbol{y}).
\end{equation}
Note that there exists $\delta\in(0,r/4)$ such that for any $\boldsymbol{x}\in B_{\delta}(\boldsymbol{z})$ and all $\boldsymbol{y}\in\Omega\setminus B_{r/2}(\boldsymbol{z})$
$$|G(\boldsymbol{x},\boldsymbol{y})|\leq\widetilde{C}:=\max\left\{\Big|\log\frac{r}{4}\Big|,\left(\frac{r}{4}\right)^{2-d}\right\},$$
which implies that for all $\boldsymbol{y}\in\Omega\setminus B_{r/2}(\boldsymbol{z})$ and $\boldsymbol{x}\in B_{\delta}(\boldsymbol{z})$
\begin{equation}\label{equbed}
|G(\boldsymbol{x},\boldsymbol{y})f_{2}(\boldsymbol{y})|\leq \widetilde{C}|f_{2}(\boldsymbol{y})|.
\end{equation}
Moreover, $f_{2}\in L^{1}(\Omega,\mu)$ as $\|f_{2}\|_{L^{1}(\Omega,\mu)}\leq\|f\|_{L^{1}(\Omega,\mu)}$. Therefore, by (\ref{eq5.15aa}), (\ref{equbed}) and dominated convergence theorem, we have
\begin{align*}
\lim_{\boldsymbol{x}\to \boldsymbol{z}}G_{\mu}f_{2}(\boldsymbol{x})&=\lim_{\boldsymbol{x}\to \boldsymbol{z}}\int_{\Omega}G(\boldsymbol{x},\boldsymbol{y})f_{2}(\boldsymbol{y})\,d\mu(\boldsymbol{y})\\
&=\lim_{\boldsymbol{x}\to \boldsymbol{z}}\int_{\Omega\setminus B_{r/2}(\boldsymbol{z})}G(\boldsymbol{x},\boldsymbol{y})f_{2}(\boldsymbol{y})\,d\mu(\boldsymbol{y})\\
&=\int_{\Omega\setminus B_{r/2}(\boldsymbol{z})}G(\boldsymbol{z},\boldsymbol{y})f_{2}(\boldsymbol{y})\,d\mu(\boldsymbol{y})\\
&=G_{\mu}f_{2}(\boldsymbol{z}).
\end{align*}
Therefore, for any $\epsilon>0$, there exists $\delta>0$ such that for any $\widetilde{\boldsymbol{z}}\in B_{\delta}(\boldsymbol{z})$, $\big|G_{\mu}f_{2}(\boldsymbol{z})-G_{\mu}f_{2}(\widetilde{\boldsymbol{z}})\big|<\epsilon$. Combining this, step 1 and the definition of $f_{1}$ and $f_{2}$, we have
\begin{align*}
\left|G_{\mu}f(\boldsymbol{z})-G_{\mu}f(\widetilde{\boldsymbol{z}})\right|&=\left|G_{\mu}f_{1}(\boldsymbol{z})+G_{\mu}f_{2}(\boldsymbol{z})-G_{\mu}f_{1}(\widetilde{\boldsymbol{z}})-G_{\mu}f_{2}(\widetilde{\boldsymbol{z}})\right|\\
&\leq\left|G_{\mu}f_{1}(\boldsymbol{z})\right|+\left|G_{\mu}f_{1}(\widetilde{\boldsymbol{z}})\right|+\left|G_{\mu}f_{2}(\boldsymbol{z})-G_{\mu}f_{2}(\widetilde{\boldsymbol{z}})\right|\\
&\leq3\epsilon,
\end{align*}
which shows that $G_{\mu}f$ is continuous at $\boldsymbol{z}$. Since $\boldsymbol{z}$ is arbitrary, we obtain that $G_{\mu}f$ is continuous on $\Omega$.

\noindent{\em Step 2.} Since $\Omega$ has Lipschitz boundary, $\partial\Omega$ is compact and $u=0$ on $\partial\Omega$, by the proof of Proposition \ref{pro3.1}, we can choose $r>0$ sufficiently small and a finite collection of $r$-balls $B_{r}(\boldsymbol{x})$ such that
$$\partial\Omega\subset\bigcup_{i=1}^{N}B_{r}(\boldsymbol{x}_{i}),\quad \boldsymbol{x}_{i}\in\partial\Omega,$$
and $G_{\mu}u$ is continuous on each $B_{r}(\boldsymbol{x}_{i})\cap\overline{\Omega}$. We can choose $V\subset\subset\Omega$ such that
$$V\cap B_{r}(\boldsymbol{x}_{i})\neq\emptyset \qquad\text{and}\qquad
\overline{\Omega}=V\cup\bigcup_{i=1}^{N}\Big(B_{r}(\boldsymbol{x}_{i})\cap\overline{\Omega}\Big).$$
Combining this and the conclusion of Step 1 completes the proof.
\end{proof}

\section{ Examples of continuous eigenfunctions }\label{sec6}

In this section, we assume $\Omega\subset\R^{2}$. We will construct some examples of continuous eigenfunctions of $\Delta_{\mu}$. As mentioned in Section 1, we are interested in the case that the measures are singular with respect to Lebesgue measure.

Let $\mathcal{D}(\R^{d}):=\{v(\boldsymbol{x}):\R^{d}\to\C|v(\boldsymbol{x})\in C^{\infty}_{0}(\R^{d})\}$. Recall that a distribution $T:\mathcal{D}(\R^{d})\to\C$ is a continuous linear map \cite{Amol_2017,Grubb_2008}.  For any local integrable function $f(\boldsymbol{x})\in L_{\rm loc}^{1}(\mathbb{R}^{d})$, define a distribution $T_{f}$ as
$$
\langle T_{f},v\rangle:=\int_{\mathbb{R}^{d}}f(\boldsymbol{x})v(\boldsymbol{x})\,d\boldsymbol{x},\qquad v(\boldsymbol{x})\in\mathcal{D}(\mathbb{R}^{d}).
$$
The $i$-th partial distributional derivative $\partial T/\partial x_{i}$ of a distribution $T$ is defined by
$$\Big\langle\frac{\partial T}{\partial x_{i}},v\Big\rangle=-\Big\langle T,\frac{\partial v}{\partial x_{i}}\Big\rangle,\qquad v(\boldsymbol{x})\in\mathcal{D}(\mathbb{R}^{d}).$$

We consider the square domain $\Omega=(-1,1)\times(-1,1)$. In order to describe the distributional derivative of functions we construct, we need to define the following specific distributions in $\mathcal{D}'(\Omega)$, where $\mathcal{D}'(\Omega)$ is the dual space of $\mathcal{D}(\Omega)$.
\begin{defi}\label{DF6.1}
Let $\Omega=(-1,1)\times(-1,1)$ and $\beta\in(-1,1)$ be a fixed constant. For any $v(x,y)\in\mathcal{D}(\Omega)$,
we define $\delta^{\beta,\rm{II}}$ and $\delta^{{\rm I},\beta}$ as distributions in $\mathcal{D}'(\Omega)$ that satisfy the following equations:
\begin{equation}\label{eq(6.1)}
\big\langle \delta^{\beta,\rm\uppercase\expandafter{\romannumeral2}},v(x,y)\big\rangle
:=\int_{\Omega}v(x,y)\,d\delta_{\beta}(x)dy=
\int_{-1}^{1}v(\beta,y)\,dy,
\end{equation}
and
\begin{equation}\label{eq(6.2)}
\big\langle \delta^{{\rm I},\beta},v(x,y)\big\rangle
:=\int_{\Omega}v(x,y)\,dx d\delta_{\beta}(y)=
\int_{-1}^{1}v(x,\beta)\,dx,
\end{equation}
where $\delta_{\beta}$ is the {\rm Dirac} measure at $\beta$ defined on $\R$.
\end{defi}

\begin{rmk}
The superscript $\beta$ of $\delta^{\beta,\rm{II}}$ and $\delta^{{\rm I},\beta}$ represents the point at which the Dirac measure takes value 1. The position of $\beta$
indicates the axis on which the Dirac measure is defined. Hence, $\delta_{\beta}$ in (\ref{eq(6.1)}) is defined on the $x$-axis, and $\delta_{\beta}$ in (\ref{eq(6.2)}) is defined on the $y$-axis. The Roman superscripts ${\rm I}$ and ${\rm II}$ represent $dx$ and $dy$, respectively.
\end{rmk}

It can be checked directly that $\delta^{{\rm I},\beta}$ and $\delta^{\beta,{\rm II}}$ are distributions. We have the following property.
\begin{prop}
Use the above notations. For any $f(x,y)\in C(\Omega)$, $f(\beta,y)\delta^{\beta,\rm\uppercase\expandafter{\romannumeral2}}$ and $f(x,\beta)\delta^{\rm\uppercase\expandafter{\romannumeral1,\beta}}$ are distributions in $\mathcal{D}'(\Omega)$.
\end{prop}
\begin{proof}
For any $v(x,y)\in\mathcal{D}(\Omega)$, by (\ref{eq(6.1)}) and (\ref{eq(6.2)}), we have
\begin{equation}\label{eq(6.3)}
\big\langle f(\beta,y)\delta^{\beta,\rm\uppercase\expandafter{\romannumeral2}},v(x,y)\big\rangle
=\int_{\Omega}f(\beta,y)v(x,y)\,d\delta_{\beta}(x)\,dy=
\int_{-1}^{1}f(\beta,y)v(\beta,y)\,dy
\end{equation}
and
\begin{equation}\label{eq(6.4)}
\big\langle f(x,\beta)\delta^{\rm\uppercase\expandafter{\romannumeral1,\beta}},v(x,y)\big\rangle
=\int_{\Omega}f(x,\beta)v(x,y)\,d\delta_{\beta}(y)\,dx=
\int_{-1}^{1}f(x,\beta)v(x,\beta)\,dx.
\end{equation}
By checking the linearity and continuity, we obtain the proposition.
\end{proof}

For the above square domain $\Omega=(-1,1)\times(-1,1)$, let $\mu_{0}$ be the 1-dimension Lebesgue measure defined on $[-1,1]\times\{0\}$ and $\mu_{1}$ be the 1-dimension Lebesgue measure defined on $\{0\}\times[-1,1]$, as shown in Figure \ref{fig1}. We will use $\mu_{0}$ and $\mu_{1}$ to construct a measure on $\Omega$, which is singular respect to the Lebesgue measure on $\R^{2}$.

\begin{exam}\label{EX1}
Use the above notation. Let $\mu=\mu_{0}+\mu_{1}$ be defined on $\Omega$. Then $\mu$ is singular respect to the $2$-dimension Lebesgue measure $d\boldsymbol{x}$. Let $\Delta_{\mu}$ be defined as in Section $2$. Then
\begin{equation}
u(x,y)=1+|xy|-|x|-|y|
\end{equation}
is a $2$-eigenfunction of (\ref{eq(2.2)}). $u(x,y)$ (see Figure \ref{fig2}) is continuous and has no nodal points in $\Omega$. Hence, it is a first eigenfunction of (\ref{eq(2.2)}).
\end{exam}
\begin{figure}[h!]
\begin{minipage}[t]{0.35\linewidth}
\centering
%Requires \usepackage{graphicx}
  \includegraphics[scale=0.45]{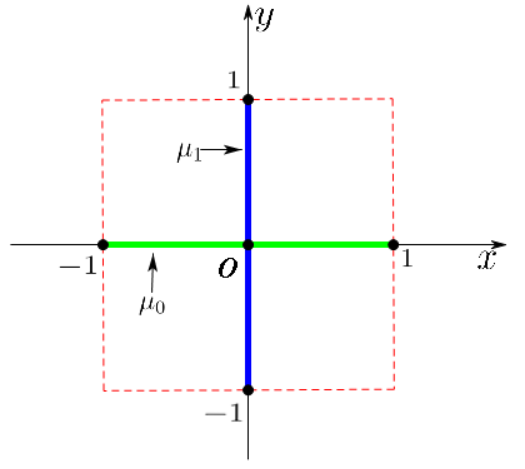}
  \caption{}\label{fig1}
  \end{minipage}
 %\hspace{3cm}
\begin{minipage}[t]{0.45\linewidth}
\centering
%Requires \usepackage{graphicx}
  \includegraphics[scale=0.35]{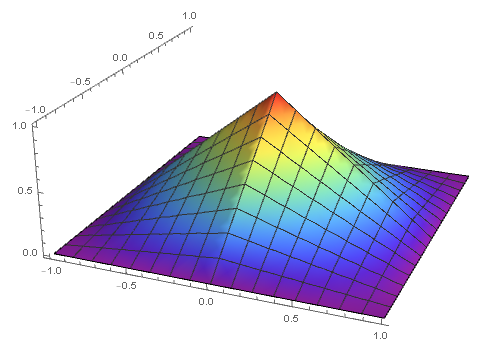}
  \caption{}\label{fig2}
  \end{minipage}
\end{figure}
\begin{proof}
According to \cite[Proposition2.2]{Hu-Lau-Ngai_2006} and (\ref{eq(2.2)}), we have, for any $v(x,y)\in\mathcal{D}(\Omega)$,
\begin{equation}\label{eq(6.1.1)}
-\int_{\Omega}\Delta u(x,y)v(x,y)\,dxdy=\lambda\int_{\Omega}u(x,y)v(x,y)\,d\mu.
\end{equation}
We first calculate the second order distributional derivative of $u(x,y)$. For any $v(x,y)\in\mathcal{D}(\Omega)$, we have
\begin{align*}
\Big\langle \frac{\partial^{2} T_{|x|}}{\partial x^{2}},v(x,y)\Big\rangle
&=-\Big\langle \frac{\partial T_{|x|}}{\partial x},v_{x}(x,y)\Big\rangle
 =\big\langle T_{|x|},v_{xx}(x,y)\big\rangle\\
&=\int_{\Omega}|x|v_{xx}(x,y)\,dxdy\\
&=\int_{-1}^{1}\,dy\int_{0}^{1}xv_{xx}(x,y)\,dx-\int_{-1}^{1}\,dy\int_{-1}^{0}xv_{xx}(x,y)\,dx\\
&=2\int_{-1}^{1}v(0,y)\,dy\\
&=2\big\langle\delta^{0,\rm\uppercase\expandafter{\romannumeral2}},v(x,y)\big\rangle.
\end{align*}
Hence,
$$\Delta|x|=2\delta^{0,\rm\uppercase\expandafter{\romannumeral2}}$$
in the sense of distribution. By the same argument, we can get
$$\Delta|y|=2\delta^{\rm\uppercase\expandafter{\romannumeral1},0}$$
and
$$\Delta|y|=2\left(|x|\delta^{\rm\uppercase\expandafter{\romannumeral1},0}+|y|\delta^{0,\rm\uppercase\expandafter{\romannumeral2}}\right)$$
in the sense of distribution. In summary,
\begin{equation}\label{eq(6.5)}
\Delta u(x,y)=\Delta(1+|xy|-|x|-|y|)=2(|x|\delta^{\rm\uppercase\expandafter{\romannumeral1},0}+|y|\delta^{0,\rm\uppercase\expandafter{\romannumeral2}}-\delta^{\rm\uppercase\expandafter{\romannumeral1},0}-\delta^{0,\rm\uppercase\expandafter{\romannumeral2}}),
\end{equation}
in the sense of distribution. By direct calculation, we have
\begin{align*}
-\int_{\Omega}\Delta u(x,y)v(x,y)\,dxdy=&-\int_{\Omega}\Delta\big(1+|xy|-|x|-|y|\big) v(x,y)\,dxdy\\
=&\,2\Big[\int_{-1}^{1}\big(1-|x|\big)v(x,0)\,dx
+\int_{-1}^{1}\big(1-|y|\big)v(0,y)\,dy\Big]\\
=&\,2\int_{\Omega}\big(1-|x|-|y|\big)v(x,y)\,d(\mu_{0}+\mu_{1})\\
=&\,2\int_{\Omega}\big(1-|x|-|y|\big)v(x,y)\,d\mu.
\end{align*}
On the other hand,
\begin{align*}
\lambda\int_{\Omega}u(x,y)v(x,y)\,d\mu
&=\lambda\int_{\Omega}\big(1+|xy|-|x|-|y|\big) v(x,y)\,d(\mu_{0}+\mu_{1})\\
&=\lambda\int_{\Omega}\big(1-|x|-|y|\big) v(x,y)\,d(\mu_{0}+\mu_{1})\\
&=\lambda\int_{\Omega}\big(1-|x|-|y|\big) v(x,y)\,d\mu.
\end{align*}
Combining the above two equalities and (\ref{eq(6.1.1)}), we get $\lambda=2$. Therefore, if $u(x,y)\in\rm{Dom}(\Delta_{\mu})$, it is a $2$-eigenfunction.
 We next prove $u(x,y)\in\rm{Dom}(\Delta_{\mu})$. We divide the proof into three steps.

\noindent{\em Step 1}. We claim that $u(x,y)\in H^{1}_{0}(\Omega)$. To see this, let
$$
\xi_{1}(x,y)=\begin{cases}
{\rm sgn}(x)(y-1),\quad &y\geq0,\\
-{\rm sgn}(x)(y+1),\quad &y<0,
\end{cases}
$$
and
$$
\xi_{2}(x,y)=\begin{cases}
{\rm sgn}(y)(x-1),\quad &x\geq0,\\
-{\rm sgn}(y)(x+1),\quad &x<0.
\end{cases}
$$
By direct calculation, $(\xi_{1}(x,y),\xi_{2}(x,y))$ is the weak derivative of $u(x,y)$. Moreover, both $\xi_{1}(x,y)$ and $\xi_{2}(x,y)$ are in $L^{2}(\Omega)$. Thus, $u(x,y)\in H^{1}(\Omega)$. Note that the square domain $\Omega$ is a Lipschitz domain. According to \cite[Theorem 3.33]{McLean_2000},
we have, $u(x,y)\in H^{1}_{0}(\Omega)$.

\noindent{\em Step 2}. We prove that $u(x,y)\in\rm{Dom}(\mathcal{E})$. Let $\mathcal{I}$ be defined as in (\ref{eq(2.1.1)}) and
$$\tau(\boldsymbol{x})\in\mathcal{N}:=\big\{u\in H^1_0(\Omega)\,\big|\,\|\mathcal{I}(u)\|_{L^2(\Omega,\mu)}=0\big\}.$$
Write $\Omega=\bigcup_{i=1}^{4}\Omega_{i}$, where
$$\Omega_{1}=[0,1)\times[0,1),~\Omega_{2}=(-1,0)\times[0,1),~\Omega_{3}=(-1,0]\times(-1,0],~\Omega_{4}=(0,1)\times(-1,0).$$
By Step 1, we have
$$
(u,\tau)_{H_{0}^{1}(\Omega)}=\int_\Omega \nabla u \cdot\nabla \tau \,d\boldsymbol{x}
=\int_{\Omega}(\xi_{1},\xi_{2})\cdot\nabla \tau \,d\boldsymbol{x}
=\sum_{i=1}^{4}\int_{\Omega_{i}}(\xi_{1},\xi_{2})\cdot\nabla \tau \,d\boldsymbol{x}.
$$
Since $\tau\in\mathcal{N}$, we have
\begin{align*}
\int_{\Omega_{1}}(\xi_{1},\xi_{2})\cdot\nabla \tau \,d\boldsymbol{x}
&=\int_{0}^{1}\int_{0}^{1}(y-1,x-1)\cdot(\tau_{x},\tau_{y})\,dxdy\\
&=\int_{0}^{1}\int_{0}^{1}(y-1)\tau_{x}+(x-1)\tau_{y}\,dxdy\\
&=\int_{0}^{1}(y-1)\,dy\int_{0}^{1}\tau_{x}\,dx
+\int_{0}^{1}(x-1)\,dx\int_{0}^{1}\tau_{y}\,dy\\
&=\int_{0}^{1}(y-1)\tau(x,y)\Big|_{0}^{1}\,dy
+\int_{0}^{1}(x-1)\tau(x,y)\Big|_{0}^{1}\,dx\\
&=-\int_{0}^{1}(y-1)\tau(0,y)\,dy-\int_{0}^{1}(x-1)\tau(x,0)\,dx\\
&=0.
\end{align*}
The last equality holds because $\tau=0$ $dx$-a.e. on $\{0\}\times(-1,1)$ and $(-1,1)\times\{0\}$. Similarly, we get
$$\int_{\Omega_{i}}(\xi_{1},\xi_{2})\cdot\nabla \tau \,d\boldsymbol{x}=0,\qquad\text{for}~i=1,2,3,4.$$
Hence, $(u,\tau)_{H^{1}_{0}(\Omega)}=0$, for $\tau\in\mathcal{N}$. Thus $u(x,y)\in\mathcal{N}^{\bot}=\rm{Dom}(\mathcal{E})$.

\noindent{\em Step 3}. We prove that $u(x,y)\in\rm{Dom}(\Delta_{\mu})$. Combining \cite[Proposition 2.2]{Hu-Lau-Ngai_2006} and (\ref{eq(6.5)}), we have
\begin{align*}
\int_{\Omega}\nabla u\cdot\nabla v\,d\boldsymbol{x}&=-\int_{\Omega}\Delta u\cdot v\,d\boldsymbol{x}\\
&=\,2\Big[\int_{-1}^{1}\big(1-|x|\big)v(x,0)\,dx
+\int_{-1}^{1}\big(1-|y|\big)v(0,y)\,dy\Big]\\
&=\,2\int_{\Omega}\big(1-|x|-|y|\big)v(x,y)\,d(\mu_{0}+\mu_{1})\\
&=\,2\int_{\Omega}\big(1+|xy|-|x|-|y|\big)v(x,y)\,d\mu.
\end{align*}
Moreover, $f(x,y):=2(1+|xy|-|x|-|y|)=2u(x,y)\in L^{2}(\Omega,\mu).$ Therefore, by \cite[Proposition 2.2]{Hu-Lau-Ngai_2006} again, $u(x,y)\in\rm{Dom}(\Delta_{\mu})$.
\end{proof}

Using the method of Example \ref{EX1}, we can construct the following two examples.

\begin{figure}[h!]
\begin{minipage}[t]{0.45\linewidth}
\centering
%Requires \usepackage{graphicx}
  \includegraphics[scale=0.5]{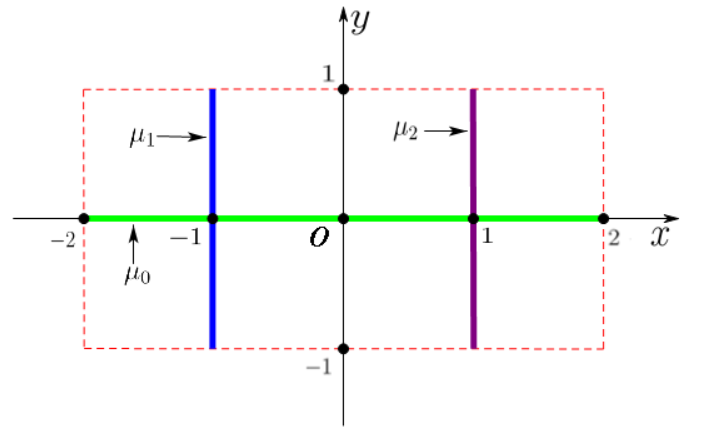}
  \caption{}\label{fig3}
  \end{minipage}
 %\hspace{3cm}
\begin{minipage}[t]{0.45\linewidth}
\centering
%Requires \usepackage{graphicx}
  \includegraphics[scale=0.4]{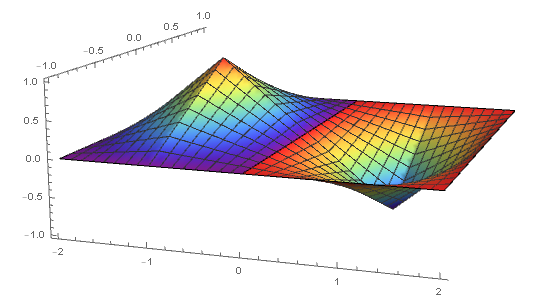}
  \caption{}\label{fig4}
  \end{minipage}
\end{figure}
\begin{exam}\label{EX2}
Let $\Omega=(-2,2)\times(-1,1)$. Write $\Omega_{1}=(-2,0]\times(-1,1)$ and $\Omega_{2}=[0,2)\times(-1,1)$.
Let measure $\mu=\mu_{0}+\mu_{1}+\mu_{2}$ be defined on $\Omega$, $\mu_{0}$, $\mu_{1}$ and $\mu_{2}$ are $1$-dimensional Lebesgue measures on $[-2,2]\times\{0\}$, $\{-1\}\times[-1,1]$ and $\{1\}\times[-1,1]$, respectively, as shown in Figure {\rm\ref{fig3}}. Then
\begin{equation}\label{EC2}
u(x,y)=\begin{cases}
1+|(x+1)y|-|x+1|-|y|, \quad\,\,\,\,\, (x,y)\in\Omega_{1},\\
-1-|(x-1)y|+|x-1|+|y|, \quad (x,y)\in\Omega_{2}
\end{cases}
\end{equation}
is a $2$-eigenfunction that satisfies equation (\ref{eq(2.2)}). $u(x,y)$ is continuous, and the nodal line of $u$ divides the domain $\Omega$ into $2$ subdomains (see Figure {\rm\ref{fig4}}).
\end{exam}
\begin{proof}
We omit the proof as it can be obtained by the argument in Example \ref{EX1}.
\end{proof}

\begin{exam}\label{EX3}
Let $\Omega=(0,2n)\times(-1,1)$, $\Omega_{i}=(2(i-1),2i]\times(-1,1),\,i=1,\ldots,n-1,\,\Omega_{n}=[2(n-1),2n)
\times(-1,1)$, and $\Omega=\bigcup_{i=1}^{n}\Omega_{i}$. Let $\mu=\mu_{0}+\mu_{1}+\cdots+\mu_{n}$ be defined on $\Omega$, where $\mu_{0}$ is the 1-dimensional Lebesgue measure on $[0,2n]\times\{0\}$, and $\mu_{i}$ is the 1-dimensional Lebesgue measure on $\{2i-1\}\times[-1,1]$, as shown in Figure {\rm\ref{fig5}}. For $i=1,\ldots,n$, Let
\begin{equation*}
\psi_{i}(x,y):=
\begin{cases}
(-1)^{i-1}\big(1+\big|\big(x-(2i-1)\big)y\big|-|x-(2i-1)|-|y|\big), \quad&(x,y)\in\Omega_{i},\\
0, \,\,\,\qquad\qquad\qquad\qquad\qquad\qquad\qquad\qquad\qquad\quad \quad &(x,y)\in\Omega\backslash\Omega_{i}.
\end{cases}
\end{equation*}
Define
\begin{equation}\label{EC3}
u(x,y):=\sum_{i=1}^{n}\psi_{i}(x,y).
\end{equation}
Then $u(x,y)$ is a $2$-eigenfunction satisfying equation (\ref{eq(2.2)}). $u(x,y)$ is continuous, and the nodal lines of the function divide $\Omega$ into $n$ subdomains (see Figure {\rm\ref{fig6}}).
\end{exam}
\begin{figure}[h!]
\centering

\begin{minipage}{0.9\linewidth}
\centering
  \includegraphics[scale=0.8,width=0.8\linewidth]{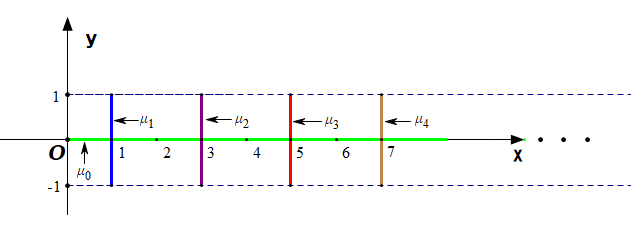}
\caption{}\label{fig5}
  \end{minipage}

\begin{minipage}{0.7\linewidth}
\centering
  \includegraphics[scale=0.8,width=0.8\linewidth]{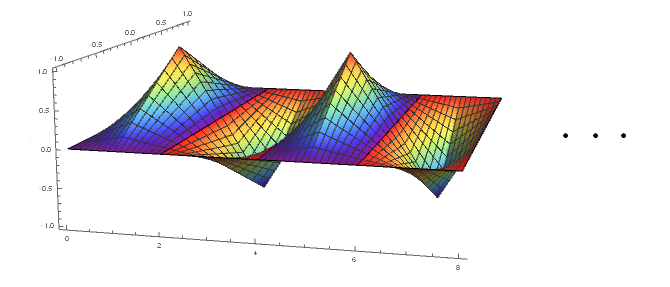}
\caption{}\label{fig6}
  \end{minipage}
\end{figure}
\begin{proof}
The method is the same as that in the proof of Example \ref{EX1}. We omit the details.
\end{proof}
\setcounter{equation}{0}

\end{document}